\documentclass[11pt, oneside]{elsarticle}   	

\usepackage{graphicx}				
\usepackage{amssymb}

\makeatletter
\def\ps@pprintTitle{%
 \let\@oddhead\@empty
 \let\@evenhead\@empty
 \def\@oddfoot{\centerline{\thepage}}%
 \let\@evenfoot\@oddfoot}
\makeatother

\usepackage{amssymb,amsmath,amsthm}
\usepackage{epsfig}
\usepackage{color}
\usepackage{makeidx}
\usepackage{bbm}
\usepackage[top=1in, bottom=1in, inner=1in, outer=1in]{geometry}
\usepackage{setspace}
\usepackage{enumitem}

\usepackage{tikz}
\usetikzlibrary{arrows,positioning,shapes,fit,calc}
\usetikzlibrary{arrows, arrows.meta}

\usetikzlibrary{hobby}

\pgfdeclarelayer{background}
\pgfsetlayers{background,main}

\newtheorem{theorem}{Theorem}
\newtheorem*{thma}{Theorem A}

\newtheorem{definition}{Definition}
\newtheorem{lemma}{Lemma}

\newtheorem*{lemmaa}{Lemma A}
\newtheorem{proposition}{Proposition}

\newtheorem*{remark}{Remark}
\newtheorem{assumption}{Assumption}

\begin{document}

\begin{frontmatter}

\title{Strong law of large numbers for a branching random walk among Bernoulli traps}

\author{Mehmet \"{O}z}
\ead{mehmet.oz@ozyegin.edu.tr}
\ead[url]{https://faculty.ozyegin.edu.tr/mehmetoz/}

\address{Department of Mathematical Engineering, \"{O}zye\u{g}in University, T\"urkiye}

\begin{abstract}
We study a $d$-dimensional branching random walk (BRW) in an i.i.d.\ random environment on $\mathbb{Z}^d$ in discrete time. A Bernoulli trap field is attached to $\mathbb{Z}^d$, where each site, independently of the others, is a trap with a fixed probability. The interaction between the BRW and the trap field is given by the hard killing rule. Given a realization of the environment, over each time step, each particle first moves according to a simple symmetric random walk to a nearest neighbor, and immediately afterwards, splits into two particles if the new site is not a trap or is killed instantly if the new site is a trap. Conditional on the ultimate survival of the BRW, we prove a strong law of large numbers for the total mass of the process. Our result is \emph{quenched}, that is, it holds in almost every environment in which the starting point of the BRW is inside the infinite connected component of trap-free sites.     
\end{abstract}

\vspace{3mm}

\begin{keyword}
Branching random walk  \sep random environment  \sep strong law of large numbers  \sep Bernoulli traps  \sep hard killing
\vspace{3mm}
\MSC[2010] 60J80 \sep 60K37 \sep 60G50 \sep 60F15 \sep 60J10
\end{keyword}

\end{frontmatter}

\pagestyle{myheadings}
\markright{BRW among Bernoulli traps\hfill}


\section{Introduction and main results}\label{section1}

We consider a discrete-time model of a branching random walk (BRW) in an i.i.d.\ random environment on the discrete lattice $\mathbb{Z}^d$. Let us refer to a lattice point as a \emph{site}. Each site $x\in\mathbb{Z}^d$ is a trap with probability $1-p\in(0,1)$ and is vacant with probability $p$ independently of the other sites. We refer to an environment formed in this way as \emph{Bernoulli traps}. The interaction between the BRW and the random environment is given by the \emph{hard killing} rule: a particle is killed and removed from the system instantaneously upon moving to a trap.  Given an environment, the dynamics of the BRW is as follows: at each time step each particle first moves according to a simple symmetric random walk on $\mathbb{Z}^d$ and then, gives birth to two offspring and dies at the new site if the new site is not a trap, or is killed without giving offspring at the new site if the new site is a trap. It is clear that the presence of traps tends to decrease the total mass, that is, the number of particles, of the BRW. Moreover, due to the possible killing of particles, given an environment $\omega$ the BRW may die out in finite time with positive $\omega$-dependent probability. Let $p_{d}$ be the critical threshold value of $p$ for site percolation in $d$ dimensions. In this paper, when $d\geq 2$ and $p>p_{d}$, we prove a quenched strong law of large numbers for the total mass of BRW among Bernoulli traps on $\mathbb{Z}^d$, conditional on the ultimate survival of the process. Here, \emph{quenched} refers to the set of environments for which the origin is inside an infinite connected component of trap-free sites.

Variations of the current problem were previously studied in a continuum setting, where the model included a branching Brownian motion (BBM) among Poissonian obstacles in $\mathbb{R}^d$. In \cite{E2008}, Engl\"ander considered a BBM in a random environment composed of spherical traps of fixed radius with centers given by a Poisson point process. The interaction between the BBM and the trap field was given by the \emph{mild} obstacle rule; that is, when a particle is inside the traps, it branches at a positive rate lower than usual and there is no killing of particles. Engl\"ander showed that on a set of full measure with respect to the Poisson point process, a kind of weak law of large numbers holds (see \cite[Theorem 1]{E2008}) for the mass of the BBM. His result was later improved by \"Oz \cite{O2021} to a strong law of large numbers, also including the case of zero branching in the trap field, and then extended by \"Oz \cite{O2024} to the setting of soft killing via a positive potential on $\mathbb{R}^d$ when $d\geq 2$. In all of the aforementioned works, the main challenge was to prove the lower bound of the law of large numbers. Here, we study a similar model in a discrete setting. The space $\mathbb{R}^d$ is replaced by the integer lattice $\mathbb{Z}^d$, continuous time is replaced by discrete time, and the BBM is replaced by a BRW. Moreover, instead of mild traps which only suppress the branching or soft traps which kill the particles at a certain rate (or with a certain probability in the discrete setting), we consider here hard traps which kill a particle immediately when the particle hits them. Compared to the cases of mild or soft traps, the proof of the lower bound is more challenging as both the system tends to produce fewer particles due to instant killing and the process lives in the unique infinite trap-free component in $\mathbb{Z}^d$, denoted by $\mathcal{C}$ here, which means possible large trap-free regions outside $\mathcal{C}$ in which the mass could otherwise grow are not accessible to the process. We refer the reader to Section~\ref{relatedmodels} for a more detailed description of related works, including the ones mentioned herein and one in which a critical BRW among Bernoulli traps is studied.

\subsection{The model}\label{section1.1}

We now present the model in detail. Firstly, we introduce the two random components, the BRW and the random trap field, and then we develop a model of random motion in a random environment by specifying an interaction between the random components.

\medskip

\textbf{1. Branching random walk:} Let $Z=(Z_n)_{n\geq 0}$ be a $d$-dimensional BRW in discrete time, where $n$ represents time. We suppose that the underlying walk is simple and symmetric and that the branching is binary. The process starts with a single particle, say at site $x\in\mathbb{Z}^d$, which at time $n=1$ first moves to a nearest neighbor of $x$ that is chosen uniformly randomly, that is, with probability $1/(2d)$, and then immediately gives birth to two offspring at this new site and dies. This means, at $n=1$, there are two particles located at one of the nearest neighbors of $x$. Starting from the site where they are born, each offspring particle behaves stochastically in the same way as their parent independently of all others and of everything that has happened in the past, and the process evolves through time in this way. The initial particle constitutes the $0$th generation, its offspring constitute the $1$st generation, and so forth. For each $n\geq 0$, $Z_n$ can be viewed as a finite discrete measure on $\mathbb{Z}^d$, which is supported at the positions of the particles at time $n$. We use $P_x$ and $E_x$, respectively, to denote the law and corresponding expectation for a BRW starting with a single particle at $x\in\mathbb{Z}^d$, and drop the subscript when $x$ is the origin. For $n\geq 0$ and a Borel set $A\subseteq\mathbb{R}^d$, we write $Z_n(A)$ to denote the mass of $Z$ that fall inside $A$ at time $n$, and set $N_n:=Z_n(\mathbb{Z}^d)$ to be the (total) mass of the BRW at time $n$.

By a \emph{simple symmetric random walk}, we mean a Markov chain $X=(X_n)_{n\geq 0}$ with state space $\mathbb{Z}^d$ and transition probabilities 
$$  \mathbf{P}\left(X_{n+1}=z \mid X_n=y \right) = \frac{1}{2d}, \quad\:\: z-y \in \left\{\pm\mathbf{e}_1, \ldots, \pm\mathbf{e}_d \right\}, $$
where $\{\mathbf{e}_j\}_{1\leq j\leq d}$ are the standard basis of unit vectors in $\mathbb{Z}^d$. We will occasionally refer to a simple symmetric random walk simply as a random walk. Throughout the paper, we denote by $X=(X_n)_{n\geq 0}$ a generic simple symmetric random walk, and use $(\mathbf{P}_x:x\in\mathbb{Z}^d)$ for the laws of random walk started at $x$ with corresponding expectations $(\mathbf{E}_x:x\in\mathbb{Z}^d)$. A branching random walk will always refer to one in discrete time with binary branching and motion component given by a simple symmetric random walk on $\mathbb{Z}^d$. 

\medskip

\textbf{2. Bernoulli traps on $\mathbb{Z}^d$:} The random environment on $\mathbb{Z}^d$ is created as follows. To each site $x\in\mathbb{Z}^d$, attach independently of all others a Bernoulli random variable $\xi(x)$ such that $\xi(x)=0$ with probability $p$ and $\xi(x)=1$ with probability $1-p$. We call a site $x$ \emph{vacant} if $\xi(x)=0$, and \emph{occupied} otherwise. Formally, the law of the trap configuration is given by the product measure $\mathbb{P}:=\mu^{\otimes\mathbb{Z}^d}$ on the sample space $\Omega:=\{0,1\}^{\mathbb{Z}^d}$ with the product $\sigma$-algebra, where $\mu$ is the probability measure on $\{0,1\}$ with $\mu(0)=p$ and $\mu(1)=1-p$. For a given $\omega\in\Omega$, we refer to the collection of all occupied sites as the \emph{trap field}. That is, the trap field is the random set of sites
$$ K(\omega):= \{x\in\mathbb{Z}^d : \xi(x)=1 \}  .  $$

\medskip

\textbf{3. Branching random walk among Bernoulli traps:} The interaction between the BRW and the random environment is given by the hard killing rule. More precisely, given a realization of the environment, that is, given an $\omega\in\{0,1\}^{\mathbb{Z}^d}$, the following happens to a particle born at time $n$ at a vacant site $x$ over one time step: it first moves according to a simple symmetric random walk to a nearest neighbor $y$ of $x$, and immediately afterwards,
\begin{itemize}
\item [$\bullet$] gives two offspring and dies if $y$ is not a trap, 
\item [$\bullet$] dies without giving offspring if $y$ is a trap. 
\end{itemize} 
We will call the model described here the \emph{BRW among Bernoulli traps on $\mathbb{Z}^d$}. Note that the current model may also be viewed as a BRW killed at the boundary of the random set $K=K(\omega)$. For $\omega\in\Omega$, we refer to $\mathbb{Z}^d$ to which $K(\omega)$ is attached simply as the random environment $\omega$, and use $P^\omega_x$ to denote the conditional law of the BRW starting with a single particle at site $x$ in the environment $\omega$. We set $P^\omega=P^\omega_0$ for simplicity.

It is known from percolation theory that in $d\geq 2$ there exists a threshold value $p_d>0$ such that when $p>p_d$, a unique infinite trap-free component exists in a.e.-environment (see \cite{AKN1987}, \cite{S1993}, and Section~\ref{sitepercol} for details). Here, a trap-free component means a connected set of trap-free sites in $\mathbb{Z}^d$. Let us denote this unique $\omega$-dependent infinite trap-free component by $\mathcal{C}$, and call $A\subseteq\mathbb{Z}^d$ \emph{accessible} in the environment $\omega$ if $A\subseteq\mathcal{C}$. Then, denoting the origin by $\mathbf{0}$, the conditions under which we may expect a law of large numbers (LLN) for the mass of the BRW to hold are as follows:
\begin{equation} \nonumber
\mathbb{P}\text{-a.s. on the set} \:\: \{\mathbf{0}\in\mathcal{C}\} \:\:  \text{and when} \:\: p> p_d .
\end{equation}
Therefore, hereafter, we assume that $p>p_d$ and a \emph{quenched} result will refer to one that holds $\mathbb{P}$-a.s.\ on the set $\{\mathbf{0}\in\mathcal{C}\}$. It can be shown that in the case of hard Bernoulli traps, even when $p>p_d$, on a set of full $\mathbb{P}$-measure the entire BRW is killed with positive $P^\omega$-probability (see Theorem~\ref{thm2}). Define for each $n\in\mathbb{N}$ the events of survival up to $n$ and ultimate survival, respectively, as
\begin{equation} \label{survival}
S_n=\{N_n\geq 1\}, \quad\quad S=\bigcap_{n\geq 1} S_n ,
\end{equation}
where $N_n=Z_n(\mathbb{Z}^d)$ is the total mass at time $n$. One sees that $\lim_{n\to\infty}P^\omega(S_n)=P^\omega(S)$ by continuity of measure from above. To obtain meaningful results, we condition the process on the event of ultimate survival. To that end, define the law $\widehat{P}^\omega$ as 
$$ \widehat{P}^\omega(\:\cdot\:):=P^\omega(\:\cdot\:\mid S). $$
The main objective of this paper is to prove a quenched strong law of large numbers (SLLN) for the mass of BRW among Bernoulli traps on $\mathbb{Z}^d$, conditional on the ultimate survival of the process.


The current model may also be viewed as branching mechanism added to the classical single-particle model of a simple symmetric random walk among Bernoulli traps. It is known from \cite{A1994} that when $d\geq 2$, $\mathbb{P}$-a.s.\ on the set $\{\mathbf{0}\in\mathcal{C}\}$, the large-time behavior of the survival probability of a single random walk among Bernoulli traps with hard killing is given by 
\begin{equation} \label{eqasymptotics}
\mathbf{P}_0^\omega(\tau_K>n) = \exp\left[-k(d,p)\frac{n}{(\log n)^{2/d}}(1+o(1)) \right] , \quad\:\: n\to\infty,
\end{equation}
where $\tau_K$ denotes the first time that the random walk hits the trap field, $\mathbf{P}_0^\omega$ is the conditional law of a random walk started at the origin in the environment $\omega$, and $k(d,p)$ is a positive constant (see \eqref{constant}). Our model inherits from the single-particle model in that $\mathbb{P}$-a.s.\ on the set $\{\mathbf{0}\in\mathcal{C}\}$ the expected total mass at time $n$ can be shown to be equal to the product of $\mathbf{P}_0^\omega(\tau_K>n)$ with $2^n$, which is the total mass at time $n$ of a BRW on $\mathbb{Z}^d$ without any traps (see Section~\ref{twocolor}). 


For quick reference, the following table collects the different probabilities that we use in this paper. The corresponding expectations will be denoted by similar fonts. In what follows, \emph{free} refers to the model where there is no trap field on $\mathbb{Z}^d$.

\begin{table}[h!]
\begin{center}
\begin{tabular}{ |c|l| } 
 \hline
 $\mathbb{P}$ & law of the Bernoulli environment on $\mathbb{Z}^d$ \\ 
 \hline
 $P_x$ & laws of free BRW started by a single particle at $x\in\mathbb{Z}^d$ \\ 
 \hline
 $P_x^\omega$ & conditional laws of BRW started by a single particle at $x$ in the environment $\omega$  \\ 
 \hline
 $\widehat{P}_x^\omega$  & $\widehat{P}_x^\omega(\:\cdot\:):=P_x^\omega(\:\cdot\: \mid S)$, where $S$ is the event of ultimate survival of the BRW \\
 \hline
 $\mathbf{P}_x$  &  laws of free simple symmetric random walk started at $x\in\mathbb{R}^d$ \\
 \hline
 $\mathbf{P}_x^\omega$  &  conditional laws of simple symmetric random walk started at $x$ in the environment $\omega$ \\
 \hline
\end{tabular}
\caption{The notation and description of the different probabilities used in this paper.}
\end{center}
\end{table}

\subsection{General framework and related models} \label{relatedmodels}

We now place our model within the general framework of branching random walks in random environments (BRWRE) for discrete-time models with state space $\mathbb{Z}^d$. The construction of the random environment in such models is based on two mechanisms: one-step transition between sites and reproduction. We now give a general formulation, essentially following the one in \cite{M2008}, where the mechanisms of transition between sites and reproduction are independent.    

Let $\mathbf{p}=(p(x,y))_{x,y\in\mathbb{Z}^d}$ be the one-step transition probabilities between sites and $\mathbf{q}=(q(x))_{x\in\mathbb{Z}^d}$ be the offspring laws, where each $q(x)=(q_k(x))_{k\in\mathbb{N}}$ is a probability measure on the nonnegative integers and $q_k(x)$ denotes the probability of giving $k$ offspring when located at site $x$. Typically, it is assumed that $p(x,y)$ depends only on $y-x$ and that there is a fixed finite set $U\subset \mathbb{Z}^d$ such that $\sum_{y-x\,\in U}p(x,y)=1$ for all $x$. In view of these assumptions, with a slight abuse of notation, let us set $p_z(x)=p(x,x+z)$, $p(x)=(p_z(x))_{z\in U}$ and $\mathbf{p}=(p(x))_{x\in\mathbb{Z}^d}$. 

We now make the environment random. For a countable set $A$, define the collection of all probability measures on $A$ as
$$ \mathcal{P}(A):=\left\{\mu=(\mu_a)_{a\in A} \in [0,1]^A : \sum_{a\in A}\mu_a =1 \right\} . $$
Choose $(p(x))_{x\in\mathbb{Z}^d}$ independently from $\mathcal{P}(U)$ according to a common distribution $P$, and choose $(q(x))_{x\in\mathbb{Z}^d}$ independently from $\mathcal{P}(\mathbb{N})$ according to a common distribution $Q$. Assume that $(p(x))_{x\in\mathbb{Z}^d}$ and $(q(x))_{x\in\mathbb{Z}^d}$ are independent of each other. In this way, an i.i.d.\ environment on $\mathbb{Z}^d$ is obtained as
$$ \omega := (\mathbf{p},\mathbf{q}) = (p(x), q(x))_{x\in \mathbb{Z}^d}  . $$

In the current model, the transition probabilities $\mathbf{p}$ are deterministic since particles move according to simple symmetric random walks independent of the position they are at, whereas the offspring laws are position-dependent and random. More precisely, the law $Q$ of the general formulation above is such that $q_2(x)=1$ with probability $p$, corresponding to vacant sites; and $q_0(x)=1$ with probability $1-p$, corresponding to occupied sites, i.e., traps. 

In many models of BRWRE, the mechanisms of transition between sites and reproduction are assumed to be independent as above. In a more general setting, these need not be independent \cite{CP2007, GMPV2010}; moreover, the transition probabilities and/or offspring laws may also depend on time as well as position \cite{Y2008, CY2011, GM2013}. 

Studies on BRWRE mainly focus on survival properties including questions of recurrence and transience, local and global growth of mass, and the extremes of the system such as the maximal displacement of particles from the starting point of the process. In \cite{M2008}, where the random environment was constructed as above, a classification of BRWRE on $\mathbb{Z}^d$ (more generally on Cayley graphs) in recurrence and transience was given. In \cite{GMPV2010}, a one-dimensional model on $\mathbb{Z}$ was studied, where the mechanisms of branching and transition between sites are not independent of each other and both are position-dependent. The process was allowed to die out, and a characterization was given for various types of survival. A similar model on $\mathbb{Z}^d$ was studied in \cite{CP2007}, where each particle gives at least one offspring so that extinction is not possible, and a dichotomy on recurrence and transience was proved as well as a shape theorem for the set of sites visited by the BRW. In \cite{Y2008}, a model on $\mathbb{Z}^d$, $d\geq 3$, in which the offpsring law $\mathbf{q}=(q(x,t):x\in\mathbb{Z}^d, t\in\mathbb{N})$ is i.i.d.\ in both space and time was studied, and several results on the local and global growth were proved including a central limit theorem for the density of the population. Later, in \cite{GM2013}, the same model was studied in the context of survival, and a survival criterion was obtained for the BRW in terms of the laws of the space-time i.i.d.\ environment. We also mention \cite{D2007}, in which the large-time asymptotics of the speed of the rightmost particle was proved for a BRW on $\mathbb{Z}$, where the reproduction law is fixed but the transition probabilities are position-dependent.

In the current work, we prove an SLLN for the total mass of the process. To the best of our knowledge, such strong laws regarding the global growth of the process are quite rare in the literature of BRWRE. One such result can be found in \cite{CY2011}, in which a model on $\mathbb{Z}^d$ where particles move according to simple symmetric random walks and reproduce according to space-time i.i.d.\ offspring distributions, was studied in the context of survival probability and local/global growth. The model was associated with directed polymers in random environment (DPRE), criteria for both local and global survival were given in terms of the free energy of the associated DPRE, and exponential growth rates for the population of BRWRE were identified, including an SLLN for the total mass of the process \cite[Theorem 2.1(b), (2.7)]{CY2011}. 


Models of BRWRE are not limited to discrete space and time. We refer the reader to \cite{K2021} for a survey of various problems and recent advances on the topic of BRWRE on $\mathbb{Z}^d$ in continuous-time, and to \cite{HL2024} and \cite{MM2019} for a model of BRWRE on $\mathbb{R}$ in discrete-time. We emphasize that the literature on the topic of BRWRE is vast, our list is far from being complete, and we have merely reviewed a selection.

We now review several models that are intimately related to the present model. The first one is about a BRW among Bernoulli traps on the discrete lattice $\mathbb{Z}^d$ similar to our model, whereas the other two are on the continuum analogue of the present model, where a branching Brownian motion (BBM) replaces a BRW and a random environment in $\mathbb{R}^d$ is created via a Poisson point process. We continue to use $\mathbb{P}$ for the law of the environment, $P^\omega$ for the law of the branching process conditional on the realization $\omega$ of the environment, and the definitions of $S_n$ and $S$ in \eqref{survival}.

\subsubsection{Mild obstacles and critical branching on $\mathbb{Z}^d$} \label{mildcritical}
In \cite{EP2017}, a critical BRW among mild Bernoulli traps on $\mathbb{Z}^d$ was studied in discrete time. In this model, given an environment $\omega$, at each time step each particle first moves according to a simple symmetric random walk and then, if the new location is a vacant site the particle either vanishes or splits into two offspring particles with equal probabilities, and if the new location is occupied by a trap then nothing happens to the particle. Hence, traps only suppress the branching of particles but not necessarily reduce the total mass since they may also prevent death. It was shown that the process ultimately dies out, that is, $P^\omega(S)=0$, in almost every environment. Then, the quenched asymptotic behavior of $P^\omega(S_n)$ was obtained as time tends to infinity. In the same paper, similar results were obtained for the case of subcritical branching, too.

\subsubsection{Mild obstacles in the continuum} \label{mildsection}
The mild obstacle problem for branching brownian motion (BBM) was studied in \cite{E2008} and \cite{O2021}. In this model, the source of randomness in the environment is a homogeneous Poisson point process $\Pi$ on $\mathbb{R}^d$, and the associated trap field is given as 
$$ K(\omega) = \bigcup_{x_i \in \text{supp}(\Pi)} \bar{B}(x_i,a)  ,$$
where $\bar{B}(x,a)$ denotes the closed ball of radius $a$ centered at $x\in\mathbb{R}^d$. The interaction between the BBM and the trap field is given by the following rule: when a particle is outside $K$, it branches at rate $\beta_2$, whereas when inside $K$, it branches at a lower rate $\beta_1$. That is, with $0\leq \beta_1<\beta_2$, under the law $P^\omega$, the branching rate is spatially dependent as
$$  \beta(x,\omega) := \beta_2\,\mathbbm{1}_{K^c(\omega)}(x) + \beta_1\,\mathbbm{1}_{K(\omega)}(x)  .  $$
Hence, there is no killing of particles but only a partial or total suppression of branching in certain random regions in $\mathbb{R}^d$. The rate $\beta_1$ was taken to be strictly positive in \cite{E2008}, and a kind of quenched weak LLN was proved for the mass of BBM among the mild obstacles (see \cite[Theorem 1]{E2008}). In \cite{O2021}, the case of $\beta_1=0$ was allowed, and the weak law proved in \cite{E2008} was improved to the strong law (see \cite[Theorem 2.1]{O2021}).

\subsubsection{Soft killing in the continuum} \label{softsection}
The soft killing problem for BBM was considered in \cite{O2024}, and can be described as follows. Let $W:\mathbb{R}^d\to (0,\infty)$ be a positive, bounded, measurable, and compactly supported \emph{killing function}, and for each realization $\omega=\sum_i \delta_{x_i}$ of the Poisson point process and $x\in\mathbb{R}^d$, define the potential
\begin{equation} 
V(x,\omega)=\sum_i W(x-x_i). \nonumber
\end{equation}
The Poissonian trap field $K(\omega)$ in $\mathbb{R}^d$ is formed as follows:
\begin{equation} 
 x\in K(\omega) \:\:\Leftrightarrow\:\:   V(x,\omega)>0  . \nonumber
\end{equation}
The \emph{soft killing} rule is that particles branch at a normal rate $\beta$ when outside $K$, whereas inside $K$ they are killed at rate $V=V(x,\omega)$ and the branching is completely suppressed. (A formal treatment of BBM killed at rate $V$ in $\mathbb{R}^d$ can be found in \cite{LV2012}.) It was shown in \cite{O2024} that under soft killing on a set of full $\mathbb{P}$-measure the entire BBM is killed with positive $P^\omega$-probability, and then, conditional on the event $S$, a kind of quenched weak LLN was proved for the mass of the BBM when $d\geq 2$  (see \cite[Theorem 2.1]{O2024}).

\subsection{Main results} \label{section2}

Let $\omega_d$ and $\lambda_d$, respectively, denote the volume of the unit ball and the principal Dirichlet eigenvalue of $-\frac{1}{2d}\Delta$ on the unit ball in $\mathbb{R}^d$. Recall that in our model, each site in $\mathbb{Z}^d$ is vacant with probability $p>0$. Define the positive constants
\begin{equation} \label{eqro}
R_0=R_0(d,p):=\left[\frac{d}{\omega_d\left(\log\frac{1}{p} \right)}\right]^{1/d}
\end{equation}
and
\begin{equation} \label{constant} 
k(d,p):=\lambda_d \left[\frac{\omega_d \left(\log\frac{1}{p} \right)}{d}\right]^{2/d}.
\end{equation}   
With these definitions, observe that $R_0=R_0(d,p)=\sqrt{\lambda_d/k(d,p)}$.

We now present our main results. Theorem~\ref{thm2} is on the survival probability of the entire BRW among Bernoulli traps under the hard killing rule, and Theorem~\ref{thm1} is a quenched SLLN for the mass of the BRW.

\begin{theorem}[Survival probability for BRW among Bernoulli traps] \label{thm2}
Consider a BRW among Bernoulli traps on $\mathbb{Z}^d$, where $d\geq 2$. Under the hard killing rule, $\mathbb{P}$-a.s.\ on the set $\{\mathbf{0}\in\mathcal{C}\}$,  
\begin{equation}
0 < P^\omega(S) < 1 . \nonumber 
\end{equation}
\end{theorem}

\medskip

Recall that the law $\widehat{P}^\omega$ is defined as $\widehat{P}^\omega(\:\cdot\:)=P^\omega(\:\cdot\:\mid S)$ and that $N_n$ denotes the total mass of the BRW at time $n$.

\begin{theorem}[Quenched SLLN for BRW among Bernoulli traps]\label{thm1}
Consider a BRW among Bernoulli traps on $\mathbb{Z}^d$, where $d\geq 2$. Under the hard killing rule, $\mathbb{P}$-a.s.\ on the set $\{\mathbf{0}\in\mathcal{C}\}$, 
\begin{equation} \label{eqthm1}
\underset{n\rightarrow\infty}{\lim} (\log n)^{2/d}\left(\frac{\log N_n}{n}-\log 2\right)=-k(d,p) \quad\:\:\:\: \widehat{P}^\omega\text{-a.s.}  
\end{equation}
\end{theorem}

\begin{remark}
Theorem~\ref{thm1} is called a law of large numbers, because it says that in a sense the mass of BRW among Bernoulli traps grows as its expectation (see Proposition~\ref{prop3}) as $n\rightarrow\infty$. That is, in a loose sense,
$$ \frac{\log N_n}{n} \:\approx \frac{\log E^\omega[N_n]}{n}, \quad n\to\infty, \quad \widehat{P}^\omega\text{-a.s.}  $$
The reason why we call it \emph{quenched} is that it holds in almost every environment where the starting point of the BRW belongs to the unique infinite trap-free cluster (see Section~\ref{sitepercol}). 

We emphasize once more that similar LLNs were shown to hold in the continuum setting for the models of BBM among mild obstacles (see \cite{E2008} and \cite{O2021}) and among obstacles with soft killing (see \cite{O2024}). 

\end{remark}

\subsection{Further notation and outline}  We introduce standard notation that will be used throughout the paper. We use $\mathbb{N}_+$ and $\mathbb{R}_+$, respectively, as the set of positive integers and positive real numbers, and set $\mathbb{N}=\mathbb{N}_+\cup\{0\}$. For any positive integer $k$, we set $\mathbb{N}_{\geq k}=\{k,k+1,k+2,\ldots\}$. For $x\in\mathbb{R}^d$, we denote by $|x|$ the Euclidean distance of $x$ to the origin. For a set $A\subseteq\mathbb{R}^d$ and $x\in\mathbb{R}^d$, we define their sum in the sense of sum of sets as $x+A:=\{x+y:y\in A\}$. For two functions $f,g:\mathbb{N}_+\to\mathbb{R}_+$, we write $g(n)=o(f(n))$ and $f(n)\sim g(n)$, respectively, if $g(n)/f(n)\rightarrow 0$ and $g(n)/f(n)\rightarrow c$ for some positive constant $c>0$ as $n\rightarrow\infty$. For an event $A$, we use $A^c$ to denote its complement, and $\mathbbm{1}_A$ its indicator function. We use $c$ as a generic positive constant, whose value may change from line to line. If we wish to emphasize the dependence of $c$ on a parameter $p$, then we write $c_p$ or $c(p)$. 

The rest of the paper is organized as follows. In Section~\ref{section3}, we collect some preparatory results for the proofs of our main results. In Section~\ref{section4}, we prove Theorem~\ref{thm2} and then construct a quenched environment in which the SLLN for BRW among Bernoulli traps will hold. This section could also be viewed as the first part of the proof of Theorem~\ref{thm1}. Then, the proof of the upper bound of Theorem~\ref{thm1} is given in Section~\ref{section5}, and the proof of the lower bound is given in Section~\ref{section6}. 


\section{Preparations}\label{section3}

In this section, we collect some preparatory results that will be needed for the proofs of our main results. 

\subsection{A two-colored branching random walk}\label{twocolor}

We will introduce a \emph{two-colored BRW}, which will be suitable for the setting of Bernoulli traps in $\mathbb{Z}^d$. Prior to that, we present some necessary formalism regarding an ordinary BRW.  

For $n\in\mathbb{N}$, let $\mathcal{N}_n$ be the set of particles of the BRW at time $n$. For $u\in\mathcal{N}_n$, let $X_u(n)$ denote its position at time $n$. For two particles $u$ and $v$, we write $v<u$ to mean that $v$ is an ancestor of $u$. Now, extend the definition of $X_u$ so that $X_u:\{0,1,\ldots,n\}\to\mathbb{Z}^d$, where 
\begin{equation} \nonumber
X_u(k) := 
\begin{cases}  
X_u(n) & \text{if} \:\: k=n, \\
X_v(k) & \text{if} \:\: v<u \:\: \text{and} \:\: v\in\mathcal{N}_k .
\end{cases}
\end{equation}
For $u\in\mathcal{N}_n$, we will refer to  
$$ X_u^{(n)}:=\left\{X_u(k):0\leq k\leq n\right\} $$
as the \emph{ancestral line (or path) of $u$ up to time $n$}. In other words, the ancestral line of $u$ up to $n$ is the set of sites visited up to $n$ by the particle $u$ and all its ancestors. The following definition is related to the notion of an ancestral line, and will be useful in the subsequent proofs.
\begin{definition}[Ancestral line segment]\label{def2}
The range of a function $\gamma:\{k\in\mathbb{N}:n_1\leq k\leq n_2\}\to\mathbb{Z}^d$ is called an \emph{ancestral line segment from time $n_1$ to time $n_2$} for a BRW if there exists $u\in\mathcal{N}_{n_2}$ such that $X_u(k)=\gamma(k)$ for all $k\in\mathbb{N}$ with $n_1\leq k\leq n_2$. 
\end{definition} 
We now introduce a two-colored BRW as follows. Let $A\subseteq\mathbb{Z}^d$ be any set of sites and let $A^c=\{x\in\mathbb{Z}^d:x\notin A\}$. For each $n\in\mathbb{N}$, color the particles in $\mathcal{N}_n$ according to the following protocol: $u\in\mathcal{N}_n$ is colored blue provided $X_u^{(n)}\subseteq A^c$; otherwise, it is colored red. Similarly, each ancestral line up to time $n$ is colored blue provided it does not hit $A$ over $[0,n]$; otherwise, it is colored blue up until it hits $A$, and then the remaining part of it is colored red. In this way, for any given set $A\subseteq\mathbb{Z}^d$, we obtain a BRW enriched by a coloring of its particles and ancestral lines with blue and red, where the coloring depends on the set $A$. Let $P^A$ be the law for the two-colored BRW, $E^A$ be the corresponding expectation, and $N_n^B$ be the number of blue particles at time $n$.

\begin{proposition}[Expected mass] \label{prop3}
Consider a BRW among Bernoulli traps on $\mathbb{Z}^d$, where $d\geq 2$. Under the hard killing rule, $\mathbb{P}$-a.s.\ on the set $\{\mathbf{0}\in\mathcal{C}\}$,
\begin{equation} \label{eqexpectedmass2}
E^\omega[N_n]=2^n\exp\left[-k(d,p)\frac{n}{(\log n)^{2/d}}(1+o(1))\right] .
\end{equation}
\end{proposition}

\begin{proof}
Let $\omega\in\{\mathbf{0}\in\mathcal{C}\}$. Observe by setting $A=K(\omega)$ in the formulation above that hard killing of a particle by $K$ under the law $P^\omega$ corresponds to changing the color of its ancestral line from blue to red (but in this case, the particle is not killed; it continues to live and branch as usual) for a two-colored BRW under the law $P^{K}$. This also implies, for each $n\in\mathbb{N}$,
$$  P^\omega(N_n=k) = P^K(N_n^B=k), \quad\:\: k=0,1,2,\ldots  $$ 
Let $\tau_K$ be the first hitting time to $K$ of a generic simple symmetric random walk $X=(X_n)_{n\geq 0}$ in $d$ dimensions, and define $\tau_K^u$ similarly for the ancestral line of particle $u\in\mathcal{N}_n$. Applying the many-to-one lemma for the BRW $\omega$-wise, from for instance \cite[Lemma 8]{HR2017} with the choices and notation $k=1$, $Y(v)=\mathbbm{1}_{\{\tau_K^v>n\}}$ and $Y=\mathbbm{1}_{\{\tau_K>n\}}$ therein and noting that $m_1(X_{p(v)})=2$ in the case of binary branching, we obtain
\begin{align}
E^\omega[N_n] &= E^K\left[N_n^B\right] = E\left[\sum_{u\in\mathcal{N}_n} \mathbbm{1}_{\{\tau_K^u>n\}} \right] \nonumber \\
&= 2^n \mathbf{E}_0 \left[\mathbbm{1}_{\{\tau_K>n\}}\right] = 2^n \mathbf{P}_0\left(\tau_K>n\right). \label{eqexpectedmass}
\end{align}
Observe that $\mathbf{P}_0\left(\tau_K>n\right)$ is the survival probability up to $n$ of a single random walk among hard Bernoulli obstacles in the environment $\omega$. Since $\omega\in\{\mathbf{0}\in\mathcal{C}\}$ by assumption, the result then follows from \eqref{eqexpectedmass} and the quenched (single-particle) survival asymptotics for a random walk among Bernoulli traps in \eqref{eqasymptotics}.  
\end{proof}

\subsection{Some large-deviation results on the mass of the branching random walk}

In this section, we state and prove two results of independent interest that concern the large-time mass of a free BRW. To this end, we introduce further notation. Denote by $B(z,r)$ the open ball of radius $r>0$ centered at $z\in\mathbb{R}^d$, and set $B_r:=B(0,r)$ for simplicity. For a Borel set $A\subseteq\mathbb{R}^d$, let $\sigma_A=\inf\{n\geq 0:X_n\notin A\}$ denote the first exit time of a random walk $X=(X_n)_{n\geq 0}$ out of $A$, and define 
\begin{equation} \nonumber
p_n(r) := \mathbf{P}_0\left(\sigma_{B_r}>n\right)    . 
\end{equation}

\begin{proposition} \label{prop1}
Let $0<c<1$ be fixed. For $r>0$, let $Y_n(r)$ denote the number of particles at time $n$ of a standard BRW whose ancestral lines up to $n$ do not exit $B_r$. Then, for every $\varepsilon>0$, there exists $R>0$ large enough such that 
\begin{equation}
P\left(Y_n(R) < c\, p_n(R) 2^n\right) \leq \varepsilon  \nonumber
\end{equation}
for all $n>R$.
\end{proposition}

\begin{proof}
We use a standard second moment argument. Throughout the proof, we suppress the dependence of $Y_n(r)$ and $p_n(r)$ on $r$ for ease of notation. It follows from the many-to-one lemma for the BRW that for each $n\in\mathbb{N}$,
\begin{equation}
E[Y_n] = p_n 2^n . \nonumber
\end{equation}
Fix $0<c<1$. Let $\text{Var}$ denote the variance associated to the law $P$. Then, by Chebyshev's inequality,
\begin{align}
P\left(Y_n < c\, p_n 2^n\right) & = P\left(E[Y_n]-Y_n >E[Y_n] - c\, p_n 2^n\right) \nonumber \\
& \leq P\left(|Y_n-E[Y_n]| > p_n 2^n - c\, p_n 2^n  \right) \nonumber \\
& \leq \frac{\text{Var}(Y_n)}{[p_n 2^n(1-c)]^2} . \label{eq1prop1}
\end{align}
Next, we estimate $\text{Var}(Y_n)$. Let $\mathcal{P}$ be the probability under which the pair of particles $(i,j)$ is chosen uniformly at random among the $2^n(2^n-1)$ possible pairs at time $n$. Denote by $\mathcal{N}_n$ as before the set of particles alive at time $n$, and for each $u\in\mathcal{N}_n$ define the event $A_u=A_u(r,n)=\{X_k(u)\in B_r\:\:\forall\: 0\leq k\leq n\}$. Let $A=A(r,n)=\{X_k\in B_r\: \forall\: 0\leq k\leq n\}$ be defined similarly for a generic random walk $(X_n)_{n\geq 0}$. Then,
\begin{align}
\text{Var}(Y_n) &= \text{Var}\left(\sum_{u\in\mathcal{N}_n} \mathbbm{1}_{A_u}\right) \nonumber \\
&= 2^n \textbf{Var} (\mathbbm{1}_A) + \sum_{1\leq i\neq j\leq 2^n} \text{Cov}\left(\mathbbm{1}_{A_i},\mathbbm{1}_{A_j}\right)  \nonumber \\
&= 2^n(p_n-p_n^2) + 2^n(2^n-1)\frac{\sum_{1\leq i\neq j\leq 2^n}\text{Cov}\left(\mathbbm{1}_{A_i},\mathbbm{1}_{A_j}\right)}{2^n(2^n-1)} \nonumber \\
&= 2^n(p_n-p_n^2) + 2^n(2^n-1)[(\mathcal{E}\otimes P)(A_i\cap A_j)-p_n^2], \label{eq2prop1}
\end{align}
where $\textbf{Var}$ denotes the variance associated to the law $\mathbf{P}_0$, $\text{Cov}$ denotes the covariance associated to the law $P$, and $(\mathcal{E}\otimes P)(A_i\cap A_j)=\mathcal{E}[P(A_i\cap A_j)]$ denotes averaging $P(A_i\cap A_j)$ over the $2^n(2^n-1)$ pairs of particles at time $n$. Let $Q^n$ be the distribution of the time at which the most recent common ancestor of $i$th and $j$th particles at time $n$ was alive under the law $\mathcal{E}\otimes P$, that is, of two particles of BRW which are chosen uniformly randomly at time $n$. By the Markov property applied at this time, we have
\begin{equation}
(\mathcal{E}\otimes P)(A_i\cap A_j) = p_n \sum_{k=0}^{n-1}\, \sum_{x\in B_r\cap\,\mathbb{Z}^d} p_{n-k,x}\,\tilde{p}(k,0,x)\, Q^n(k), \label{eq3prop1}
\end{equation}
where we have defined for each $k\geq 0$ and $x,y\in\mathbb{Z}^d$,
\begin{equation} \nonumber
p_{k,x} := \mathbf{P}_x\left(\sigma_{B_r}>k\right), \quad\quad  \tilde{p}(k,x,y) := \mathbf{P}_x\left(X_k=y \mid X_j\in B_r \:\:\forall\: 0\leq j\leq k\right).
\end{equation}
The Markov property of a random walk applied at time $k$, $0\leq k\leq n$, yields
\begin{equation}
p_n = p_k \sum_{x\in B_r\cap\,\mathbb{Z}^d} p_{n-k,x} \,\tilde{p}(k,0,x) .   \label{eq4prop1}
\end{equation}
Then, it follows from \eqref{eq3prop1} and \eqref{eq4prop1} that 
\begin{equation}
(\mathcal{E}\otimes P)(A_i\cap A_j) = p_n^2 \sum_{k=0}^{n-1} \frac{1}{p_k} Q^n(k).  \label{eq5prop1} 
\end{equation}
Define for $n\geq 1$, 
$$  J_n :=  \sum_{k=0}^{n-1} \frac{1}{p_k} Q^n(k) .$$
Then, it follows from \eqref{eq2prop1} and \eqref{eq5prop1} that for all $n\geq 1$,
\begin{equation}
\text{Var}(Y_n) \leq  2^n p_n+ 2^{2n} p_n^2(J_n-1) . \label{eq6prop1} 
\end{equation}
For $n\geq 2$ and $0< r<n$, split $J_n$ into two pieces as 
$$ J_n^{(1)} :=  \sum_{k=0}^{\lfloor r\rfloor} \frac{1}{p_k} Q^n(k),  \quad\quad J_n^{(2)} :=  \sum_{k=\lfloor r\rfloor+1}^{n-1} \frac{1}{p_k} Q^n(k) .$$
(Set $J_n^{(2)}=0$ if $n-1<\lfloor r\rfloor+1$.) Next, we bound $J_n-1$, which is clearly positive, from above. 

To bound $J_n^{(1)}-1$ from above, observe that $p_k$ is nonincreasing in $k$, which yields $J_n^{(1)}\leq (p_{\lfloor r\rfloor})^{-1}$. From a standard large-deviation estimate on linear displacements for a random walk, see for instance \cite[Proposition 2.1.2]{LL2010}, there exist positive constants $c_1$ and $c_2$ such that for all $r>0$,
\begin{equation}
\mathbf{P}_0\left(\sup_{0\leq k\leq \lfloor r\rfloor}|X_k|>r\right) = 1-p_{\lfloor r\rfloor} \leq c_1 e^{-c_2 r} . \nonumber
\end{equation}
It then follows that for all $r$ large enough and $n>r$,
\begin{equation}
J_n^{(1)}-1 \leq \frac{c_1 e^{-c_2 r}}{1-c_1 e^{-c_2 r}} \leq 2 c_1 e^{-c_2 r}. \label{eq7prop1}  
\end{equation}
To bound $J_n^{(2)}$ from above, we use the explicit form of the distribution $Q^n$ from \cite[Section 4.1]{GH2018}:
\begin{equation}
Q^n(k) = \frac{2^{-k-1}}{1-2^{-n}}, \quad\:\: k=0,1,\ldots,n-1 .  \label{eq8prop1} 
\end{equation}
Let $\lambda_d$ be, as before, the principal Dirichlet eigenvalue of $-\frac{1}{2d}\Delta$ on the unit ball in $\mathbb{R}^d$. By Corollary 6.9.6 and Proposition 8.4.2 in \cite{LL2010}, there exists $c_3>0$ such that
\begin{equation} \label{eq9prop1} 
p_k = \mathbf{P}_0\left(\sigma_{B_r}>k\right)   \geq c_3 \exp\left[-\frac{\lambda_d k}{r^2}(1+o(1))\right], \quad r\to\infty .  
\end{equation}
Then, it follows from \eqref{eq8prop1} and \eqref{eq9prop1} that for all $r$ large enough and $n>r$,
\begin{align}
J_n^{(2)} &\leq \sum_{k=\lfloor r\rfloor+1}^{n-1} c_3^{-1} e^{\frac{2\lambda_d k}{r^2}}2^{-k} \:\:\leq \sum_{k=\lfloor r\rfloor+1}^\infty c_3^{-1} e^{-k\left(\log 2-\frac{2\lambda_d}{r^2} \right)} \nonumber \\
&\leq \sum_{k=\lfloor r\rfloor+1}^\infty c_3^{-1} e^{-k\left(\frac{1}{2}\log 2\right)} = c_3^{-1}\frac{e^{-(\lfloor r\rfloor+1)\frac{1}{2}\log 2}}{1-e^{-\frac{1}{2}\log 2}}, \label{eq10prop1} 
\end{align}
where we have chosen $r$ so large that $\frac{2\lambda_d}{r^2}<\frac{1}{2}\log 2$ in the third inequality. Then, by \eqref{eq7prop1} and \eqref{eq10prop1}, we have
\begin{equation}
J_n-1 \leq  2 c_1 e^{-c_2 r} + 4 c_3^{-1} e^{-\left(\frac{1}{2}\log 2\right)r} .  \label{eq11prop1} 
\end{equation}
Let $c_4=2\max\{2c_1,4c_3^{-1}\}$ and $c_5=\min\{c_2,(\log 2)/2\}$. Then, for all $r$ large enough and $n>r$, it follows from \eqref{eq6prop1} and \eqref{eq11prop1} that 
\begin{equation}
\text{Var}(Y_n) \leq  2^n p_n+ 2^{2n} p_n^2 c_4 e^{-c_5 r}, \nonumber
\end{equation}
which, by \eqref{eq1prop1}, implies that
\begin{equation}
P\left(Y_n < c\, p_n 2^n\right) \leq \frac{2^n p_n+ 2^{2n} p_n^2 c_4 e^{-c_5 r}}{[p_n 2^n(1-c)]^2} = \frac{2^{-n}}{p_n(1-c)^2} + \frac{c_4 e^{-c_5 r}}{(1-c)^2} .  \label{eq12prop1} 
\end{equation}  
It is clear due to \eqref{eq9prop1} that for every $\varepsilon>0$, there exists $r=R$ large enough such that for all $n>R$ the right-hand side of \eqref{eq12prop1} is at most $\varepsilon$. 
\end{proof}

The following result is an extension of Proposition~\ref{prop1}, where the constant radius $r$ is replaced by a time-dependent radius $r(n)$.

\begin{proposition} \label{prop4}
Let $r:\mathbb{N}_+\to\mathbb{R}_+$ be increasing such that $r(n)\to\infty$ as $n\to\infty$ and $r(n)=o(\sqrt{n})$. Also, let $\gamma:\mathbb{N}_+\to\mathbb{R}_+$ be defined by $\gamma_n:=\gamma(n)=e^{-\kappa r(n)}$, where $\kappa>0$ is a constant. For $n\geq 1$, set $B_n=B(0,r(n))$, $p_n=p_n(r(n))=\mathbf{P}_0\left(\sigma_{B_n}>n\right)$, and let $Y_n$ denote the number of particles at time $n$ of a standard BRW whose ancestral lines up to time $n$ do not exit $B_n$. Then, there exists $c>0$ (independent of $\kappa$) such that for all $n$ large enough,
\begin{equation}
P\left(Y_n< \gamma_n p_n 2^n\right) \leq e^{-c r(n)}  .  \nonumber
\end{equation}
\end{proposition}

\begin{proof}
We keep the notation of the proof of Proposition~\ref{prop1}, and modify the proof where needed due to the replacement of the constants $r$ and $c$ therein, respectively, by $r(n)$ and $\gamma(n)$. Similar to \eqref{eq1prop1}, Chebyshev's inequality now yields
\begin{equation}
P\left(Y_n < \gamma_n p_n 2^n\right) \leq \frac{\text{Var}(Y_n)}{[p_n 2^n(1-\gamma_n)]^2}, \label{eq1prop4}
\end{equation}
whereas \eqref{eq2prop1} holds if we replace $r$ by $r(n)$ in the definition of $A_u$. Define for each $k\geq 0$ and $x,y\in\mathbb{Z}^d$,
\begin{equation} \nonumber
p_{k,x}^n := \mathbf{P}_x\left(\sigma_{B_n}>k\right), \quad\quad  \tilde{p}^{(n)}(k,x,y) := \mathbf{P}_x\left(X_k=y \mid X_j\in B_n \:\:\forall\: 0\leq j\leq k\right).
\end{equation}
Then, similar to \eqref{eq5prop1}, the Markov property of a random walk yields
\begin{equation}
(\mathcal{E}\otimes P)(A_i\cap A_j) = p_n^2 \sum_{k=0}^{n-1} \frac{1}{p_k^n} Q^n(k),  \label{eq2prop4} 
\end{equation}
where we have set $p_k^n = p_{k,0}^n$. Next, for $n$ large enough so that $r(n)<n-1$, define
$$ J_n^{(1)} :=  \sum_{k=0}^{\lfloor r(n)\rfloor} \frac{1}{p_k^n} Q^n(k),  \quad\quad J_n^{(2)} :=  \sum_{k=\lfloor r(n)\rfloor+1}^{n-1} \frac{1}{p_k^n} Q^n(k), \quad\quad J_n:=J_n^{(1)}+J_n^{(2)}.   $$
By the same argument leading to \eqref{eq7prop1}, there exist $c_1>0$ and $c_2>0$ such that \mbox{for all $n$ large enough,}
\begin{equation}
J_n^{(1)}-1 \leq \frac{c_1 e^{-c_2 r(n)}}{1-c_1 e^{-c_2 r(n)}} \leq 2 c_1 e^{-c_2 r(n)}. \label{eq3prop4}  
\end{equation}
By Corollary 6.9.6 and Proposition 8.4.2 in \cite{LL2010}, there exists $c_3>0$ such that
\begin{equation} \label{eq5prop4}
p_k = \mathbf{P}_0\left(\sigma_{B_n}>k\right)   \geq c_3 \exp\left[-\frac{\lambda_d k}{(r(n))^2}(1+o(1))\right], \quad n\to\infty .  
\end{equation}
Then, since $r(n)\to\infty$ as $n\to\infty$ by assumption, by the same argument leading to \eqref{eq10prop1}, for all $n$ large enough,
\begin{equation} \label{eq4prop4} 
J_n^{(2)} \leq c_3^{-1}\frac{e^{-(\lfloor r(n)\rfloor+1)\frac{1}{2}\log 2}}{1-e^{-\frac{1}{2}\log 2}}.
\end{equation}
Letting $c_4=2\max\{2c_1,4c_3^{-1}\}$ and $c_5=\min\{c_2,(\log 2)/2\}$ as before, it follows from \eqref{eq2prop1} and \eqref{eq2prop4}-\eqref{eq4prop4} that for all large $n$,
\begin{equation}
\text{Var}(Y_n) \leq  2^n p_n+ 2^{2n} p_n^2 c_4 e^{-c_5 r(n)}.  \nonumber
\end{equation}
Then, by \eqref{eq1prop4}, for all large $n$,
\begin{equation} \nonumber
P\left(Y_n < \gamma_n p_n 2^n\right) \leq \frac{2^{-n}}{p_n(1-\gamma_n)^2} + \frac{c_4 e^{-c_5 r(n)}}{(1-\gamma_n)^2} .\end{equation} 
To complete the proof, observe that $(p_n)^{-1}\to \infty$ at most subexponentially fast in $n$ due to \eqref{eq5prop4} since $r(n)\to\infty$ as $n\to\infty$ by assumption, and that $(1-\gamma_n)^{-2}\to 1$ as $n\to\infty$ since $\gamma_n\to 0$ as $n\to\infty$ by assumption. Moreover, $c_5=\min\{c_2,(\log 2)/2\}$ has no dependence on $\kappa$.
\end{proof}

The term `overwhelming probability' is henceforth used with a precise meaning, which is given as follows.
\begin{definition}[Overwhelming probability]
Let $(A_n)_{n>0}$ be a family of events indexed by time $n$, and $\mathcal{P}$ be the relevant probability. We say that $A_n$ occurs \emph{with overwhelming probability} if 
$$\underset{n\to\infty}{\lim}\mathcal{P}(A_n^c)=0 . $$ 
\end{definition}

\section{Almost sure environment}\label{section4}

In this section, we will establish a \emph{quenched environment} in which the SLLN will hold. That is, we will construct a set of environments $\Omega_0$ such that $\mathbb{P}(\Omega_0\cap\{\mathbf{0}\in\mathcal{C}\})/\mathbb{P}(\mathbf{0}\in\mathcal{C})=1$. This will be achieved in four steps. The first step chooses environments in which the vacant sites have suitable percolation properties. The second step chooses environments in which the BRW has a positive probability of ultimate survival. In the last two steps, environments with sufficiently large and sufficiently many clearings close enough to the origin are chosen: we show that $\mathbb{P}$-a.s., for all large $n$, there is at least one `huge' accesible clearing in $[-An,An]^d$ and many `large' accessible clearings (see Definitions~\ref{def1}-\ref{huge}) that are homogeneously spread in $[-An,An]^d$ at time $n$, where $A>0$ is a suitable constant. We assume that $d\geq 2$ in the remainder of this work.


\subsection{Site percolation and Bernoulli traps} \label{sitepercol}

We review some standard results about the site percolation model on the integer lattice $\mathbb{Z}^d$, which may also be viewed in the framework of Bernoulli traps on $\mathbb{Z}^d$. Let $\{\xi(x):x\in\mathbb{Z}^d\}$ be a field of i.i.d.\ random variables with corresponding probability $\mathcal{P}$ and expectation $\mathcal{E}$, and assume that each $\xi(x)$ has the distribution 
\begin{equation}
\mathcal{P}(\xi(x)=0)=p, \quad \:\: \mathcal{P}(\xi(x)=1)=1-p, \quad \:\: p\in[0,1]. \nonumber  
\end{equation}
We call a site $x$ \emph{vacant} if $\xi(x)=0$ and \emph{occupied} otherwise. For $x\in\mathbb{Z}^d$, write $x=(x_1,\ldots,x_d)$ to denote its components. For vacant sites, let two sites $x$ and $y$ in $\mathbb{Z}^d$ be adjacent if they are nearest neighbors, that is, if $\sum_{i=1}^d|x_i-y_i|=1$; whereas occupied sites $x$ and $y$ are considered adjacent if $1\leq\sum_{i=1}^d|x_i-y_i|\leq 2d$. For each $x\in\mathbb{Z}^d$, let $W(x)$ be the connected component of the site $x$ among occupied sites, which we call the \emph{occupied cluster of $x$}. That is, $W(x)$ is the set of sites which can be reached from $x$ by a path which only passes through occupied sites. If $x$ is vacant, then we set $W(x)=\emptyset$. It is known from percolation theory that when $d\geq 2$ there exists $0<p_d<1$ such that for $p>p_d$ the following holds: 
\begin{enumerate}[label=(A\arabic*)]
    \item\label{a1} Almost surely, there is a unique infinite component $\mathcal{C}$ of vacant sites.    
    \item\label{a2} Almost surely, the sets $W(x)$ are all finite and there exists a positive constant $h$ such that $\mathcal{E}[\exp\{h|W(x)|\}]<\infty$ for all $x\in\mathbb{Z}^d$. 
\end{enumerate}
The results \ref{a1} and \ref{a2} can be found, respecitvely, in Aizenman et al.\ \cite[Proposition 5.3]{AKN1987} and  Kesten \cite[Theorem 5.1]{K1982}. For $x,y\in\mathcal{C}$, denote by $d_{\mathcal{C}}(x,y)$ the distance between $x$ and $y$ in $\mathcal{C}$, that is, the minimal length of a nearest neighbor path joining $x$ and $y$ in $\mathcal{C}$. Define the norm $||\cdot||$ by $||x||=\sum_{i=1}^d|x_i|$ on $\mathbb{Z}^d$. The following result was proved in G\"artner and Molchanov \cite{GM1990}, and is critical for the proof of Theorem~\ref{thm1}.  
\begin{thma}[Distance estimate inside the infinite cluster; Lemma 2.8, \cite{GM1990}]
Assume that \ref{a1} and \ref{a2} hold. Then, there exists $\psi>1$ such that almost surely 
\begin{equation} \label{distanceestimate}
\underset{|y|\to\infty,\:y\in\mathcal{C}}{\limsup}\,\frac{d_{\mathcal{C}}(x,y)}{||x-y||}\leq \psi \quad \:\: \text{for all $x\in\mathcal{C}$}. 
\end{equation}
\end{thma}

Recall the space $\Omega$ and the law $\mathbb{P}$ from Section~\ref{section1.1}. By setting $\mathcal{P}=\mathbb{P}$, it is clear that all of the results above apply to the setting of Bernoulli traps on $\mathbb{Z}^d$.

\begin{assumption} \label{assumption1}
For the rest of the manuscript, we assume that $p>p_d$ so that $\mathbb{P}(\Omega_1)=1$, where 
\begin{equation}
\Omega_1:=\{\omega\in\Omega:\omega\:\:\text{fulfills the conditions in \emph{(A1)}, \emph{(A2)} and \eqref{distanceestimate}}\}. \label{eqomega1}   
\end{equation}  
\end{assumption}

\subsection{Survival probability} \label{survivalprob}

We first give a definition for certain random subsets of $\mathbb{Z}^d$ in the environment $\omega$. Then, we prove Theorem~\ref{thm2}. Recall from \eqref{survival} the definitions of the event of survival up to $n$ and the event of ultimate survival. 

\begin{definition}[Clearings] \label{def1}
A set of sites $A\subseteq\mathbb{Z}^d$ is called a \emph{clearing} in the random environment $\omega$ if $A\subseteq K^c$. A \emph{clearing of radius $r$} refers to the set of sites in a Euclidean ball of radius $r$ in $\mathbb{R}^d$ if all the sites in this ball are vacant. A clearing is called \emph{accessible} if all the sites within are part of the infinite trap-free component $\mathcal{C}=\mathcal{C}(\omega)$. 
\end{definition}

\begin{proof}[Proof of Theorem~\ref{thm2}]
Since $\lim_{n\to\infty}P^\omega(S_n)=P^\omega(S)$, for the lower bound it suffices to prove that $\mathbb{P}$-a.s.\ on the set $\{\mathbf{0}\in\mathcal{C}\}$, there exists a constant $c_1=c_1(\omega)$ such that $P^\omega(S_n)\geq c_1>0$ for all large $n$. 

Recall $R_0$ from \eqref{eqro}, let for $\ell>e$,
$$ R_\ell := R_0(\log\ell)^{1/d}-1-\sqrt{d},  \quad \quad C_\ell := \left[-2R_0\ell(\log\ell)^{1/d},2R_0\ell(\log\ell)^{1/d}\right]^d, $$
and define 
\begin{equation} \nonumber
\widetilde{\Omega} = \{\omega\in\Omega: \exists\,\ell_0=\ell_0(\omega),\: \forall\,\ell\geq\ell_0,\: C_\ell \:\text{contains an accessible clearing of radius $R_\ell$} \}  .
\end{equation}
It follows from \cite[Lemma 3.3]{S1993} that $\mathbb{P}(\widetilde{\Omega})=1$ (see also Section~\ref{asclearings}). Moreover, by Proposition~\ref{prop1}, there exists $R_{c}>0$ large enough such that for all $n>R_c$ the probability that at least one ancestral line of a free BRW has not left $B(0,R_c)$ up to time $n$ is at least $1/2$. 

Let $\omega\in\widetilde{\Omega}\cap\Omega_1\cap\{\mathbf{0}\in\mathcal{C}\}$ and choose $R=R(\omega)$ so that 
$$ R>R_c \quad \text{and} \quad e^{(2R/R_0)^d}>\ell_0  .$$
Then, in the random environment $\omega$, by definition of $R_\ell$ and $\widetilde{\Omega}$, the box 
$$C(\omega,d):=\left[-4R e^{(2R/R_0)^d}, 4R e^{(2R/R_0)^d}\right]^d$$
contains an accessible clearing of radius $R$. Let $B(x_0,R)$ be this accessible clearing. We may take $x_0$ to be in $\mathbb{Z}^d$. Recall that for $x,y\in\mathcal{C}$, $d_\mathcal{C}(x,y)$ denotes the minimal length of a nearest neighbor path joining $x$ and $y$ such that the entire path is in $\mathcal{C}$. Define the number
$$ N :=  \bigg\lceil 8\psi d R e^{(2R/R_0)^d}  \bigg\rceil  ,$$ 
where $\psi$ is as in \eqref{distanceestimate}. Let $n$ be so large that $n>R+N$ and consider the following survival strategy for the BRW over the interval $[0,n]$: an ancestral line (starting at the origin) follows the shortest nearest neighbor path between $\mathbf{0}$ and $x_0$ within $\mathcal{C}$, arrives at $x_0$ at time $d_{\mathcal{C}}(\mathbf{0},x_0)$, and then the sub-BRW initiated at this space-time point has at least one ancestral line segment that never leaves $B(x_0,R)$ over the interval $[d_{\mathcal{C}}(\mathbf{0},x_0),n]$. Since $\omega\in\Omega_1$, by \eqref{distanceestimate} and due to $||x_0||\leq 4 d R e^{(2R/R_0)^d}$, the probability of the partial strategy over $[0,d_{\mathcal{C}}(\mathbf{0},x_0)]$ is at least $(2d)^{-N}$. On the other hand, since $n>R+N$, we have $n-d_{\mathcal{C}}(\mathbf{0},x_0)>R$. This, along with the choice $R>R_c$ and since $B(x_0,R)$ is a clearing, implies due to Proposition~\ref{prop1} that the probability of the partial strategy over $[d_{\mathcal{C}}(\mathbf{0},x_0),n]$ is at least $1/2$. Then, by the Markov property applied at time $d_{\mathcal{C}}(\mathbf{0},x_0)$, we conclude that $P^\omega(S_n)\geq (1/2)(2d)^{-N}$ for all large $n$. 

We now show that $P^\omega(S)<1$ for a.e.-$\omega$. Let $NN(d,k)$ be the set of $k$-nearest neighbors of the origin in $d$ dimensions, where $x$ is defined to be a $k$-nearest neighbor of the origin if $||x||=\sum_{i=1}^d|x_i|=k$. Clearly, $|NN(d,k)|\geq 1$ for any $d\geq 1$ and $k\geq 1$. (The precise formula is also known: $|NN(d,k)|=\sum_{j=1}^{k\wedge d}\binom{d}{j}\binom{k-1}{j-1}2^j$.) For $n\geq 1$, define the events
$$ A_n=\{\text{all}\:\:x\in NN(d,k)\:\:\text{are vacant}\:\:\text{for all}\:\:k\leq n\} , $$
and the event
\begin{equation} \label{eqnearestneighbor}
\Omega_0 = \{\omega\in\Omega: \exists\,n_0=n_0(\omega)\:\:\text{such that}\:\:\exists\:x\in NN(d,k)\cap K(\omega)\:\:\text{for some $k\leq n_0$}\} .
\end{equation}
It follows from $\sum_{k=1}^n|NN(d,k)|\geq n$ that $\sum_{n=1}^\infty \mathbb{P}(A_n)\leq \sum_{n=1}^\infty p^n<\infty$, which in turn implies by the Borel-Cantelli lemma that $\mathbb{P}(\Omega_0)=1$. Now let $\omega\in\Omega_0$ and $n_0$ be as in \eqref{eqnearestneighbor}. In the environment $\omega$, let $y$ be a closest occupied site to the origin, $\pi(\mathbf{0},y)$ be the corresponding path of sites connecting the origin to $y$, and $k$ be the length of this path. Since $\omega\in\Omega_0$, we know that $k\leq n_0$. Then, by considering the event that each particle, starting with the initial particle, jumps along $\pi(\mathbf{0},y)$ towards $y$ right after birth, one sees that 
$$ P^\omega(S^c) \geq \left(\frac{1}{2d}\right)^{1+2+4\ldots+2^{k-1}} \geq \left(\frac{1}{2d}\right)^{2^{n_0}} > 0 .$$
This completes the proof of the upper bound of Theorem~\ref{thm2}. 
\end{proof}

For the rest of the manuscript, let $\Omega_2\subseteq\Omega$ be defined as
\begin{equation} \label{eqomega2} 
\Omega_2:=\{\omega\in\Omega:\, P^\omega(S)>0\}.   
\end{equation}

\subsection{Almost sure clearings} \label{asclearings}

We start by introducing some terminology that will be used throughout the proof. For two functions $f,g:\mathbb{N}_+\to\mathbb{R}_+$, we use $f(n)\asymp g(n)$ to mean $f(n)/g(n)\to 1$ as $n\to\infty$. Recall the definitions of $R_0$ and a \emph{clearing} from, respectively, \eqref{eqro} and Definition~\ref{def1}. 

\begin{definition}[Large clearing]\label{largeclearing}
A clearing of radius $R=R(n)$, where $R(n)\asymp R_1(\log n)^{1/d}$ and $0<R_1<R_0$, is called a \emph{large clearing in time $n$}. 
\end{definition}

\begin{definition}[Huge clearing]\label{huge}
A clearing of radius $R=R(n)$, where $R(n)\asymp R_0(\log n)^{1/d}$, is called a \emph{huge clearing in time $n$}.
\end{definition} 

For a suitably increasing function $\rho:\mathbb{N}_+\to\mathbb{R}_+$ with $\lim_{n\to\infty}\rho(n)=\infty$ and any $R_1$ with $0<R_1<R_0$, we will show that in a.e.-environment there exists $n_0=n_0(\omega)$ such that for every $n\geq n_0$, any cube of side length $\rho(n)$ within $[-An,An]^d$ contains a clearing of radius $\asymp R_1(\log\rho(n))^{1/d}$ and the cube $[-An,An]^d$ contains at least one clearing of radius $\asymp R_0(\log n)^{1/d}$ (see Figure~\ref{figure1}). We will use both \cite[Lemma 3.3]{S1993} and a variation of \cite[Lemma 3.3]{S1993} which we will state and prove. To this end, we first set forth some relevant notation and background, taken from \cite{S1993}.

For $x\in\mathbb{Z}^d$, introduce the cube in $\mathbb{R}^d$ of side length $1$ centered at $x$ as
$$ C(x) = x + \bigg[-\frac{1}{2},\frac{1}{2}\bigg)   .  $$  
Then, for $n\geq 1$, define the cube
\begin{equation} \label{eqtn}
\mathcal{T}_n = \left[-2R_0(\log 2^n)^{1/d}2^n,2R_0(\log 2^n)^{1/d}2^n\right]^d  , 
\end{equation}
and for $k\in[-2^n,2^n-1]^d \cap \mathbb{Z}^d$, define the smaller cubes within $\mathcal{T}_n$ as
$$ C_{n,k} = \left\{z\in\mathbb{R}^d : k_i2R_0(\log 2^n)^{1/d}<z_i<(k_i+1)2R_0(\log 2^n)^{1/d} \right\}. $$
For $n\geq 1$, let 
$$ R_n:=R_0(\log 2^n)^{1/d}-1  .$$
For $n\geq 1$ large enough so that $R_n>0$ and $k\in[-2^n,2^n-1]^d \cap \mathbb{Z}^d$, denote by $x(n,k)$ the unique site in $\mathbb{Z}^d$ such that $C(x(n,k))$ contains the center of $C_{n,k}$, and define the collection of sites 
\begin{equation} \label{eqsnm}
S_{n,k} = \{x\in\mathbb{Z}^d : C(x)\subseteq \bar{B}(x(n,k),R_n)\} 
\end{equation}
and the events
\begin{equation} \label{eqbnm}
\mathcal{B}_{n,k} = \{\omega\in\Omega: S_{n,k} \subseteq \mathcal{C}(\omega)\}. 
\end{equation}
Recall that $\mathcal{C}=\mathcal{C}(\omega)$ stands for the unique infinite component of vacant sites. Observe that $S_{n,k}$ is completely contained in $C_{n,k}$ and that
\begin{equation}
\text{vol}(B(0,R_n-\sqrt{d})) \leq |S_{n,k}| \leq \text{vol}(B(0,R_n)),
\end{equation}
where $\text{vol}(A)$ denotes the volume of a set $A\subseteq\mathbb{R}^d$.

\begin{lemmaa}[A.s.-existence of a huge clearing; \cite{S1993}]
$\mathbb{P}$-a.s., for all large $n$, $\mathcal{B}_{n,k}$ occurs for some $k$ in $[-2^n,2^n-1]^d\cap \mathbb{Z}^d$.
\end{lemmaa}

(The proof of Lemma A is identical to that of \cite[Lemma 3.3]{S1993} with the replacement of the Poisson intensity $\nu$ therein by $-\ln p$.) Observe that Lemma A implies in particular that $\mathbb{P}$-a.s., $\mathcal{T}_n$ contains an accessible clearing of radius $R_n-\sqrt{d}$ for all large $n$. Next, we state and prove a variation of Lemma A.

\begin{lemma}[A.s.-existence of many large clearings] \label{lemma1}
For $n\geq 1$, let $x_1,\ldots,x_{n^3}$ be any set of $n^3$ points in $\mathbb{Z}^d$, and for $1\leq j\leq n^3$, define the cubes
$$ \mathcal{T}_n^j := x_j+\mathcal{T}_n = x_j+\left[-2R_0(\log 2^n)^{1/d}2^n,2R_0(\log 2^n)^{1/d}2^n\right]^d  . $$
Then, $\mathbb{P}$-a.s., for all large $n$, each $\mathcal{T}_n^j$ for $1\leq j\leq n^3$ contains an accessible clearing of radius 
\begin{equation} \label{eqrn2}
r_n := (2R_0/3)(\log 2^n)^{1/d} . 
\end{equation}
\end{lemma}

\begin{proof}
Let $x_1,x_2,\ldots$ be a sequence of points in $\mathbb{Z}^d$, and $\mathcal{T}_n^j := x_j+\mathcal{T}_n$ for $j=1,2,\ldots$ Let $A_n^j$ and $E_n$ be the events that there is an accessible clearing of radius $r_n$ in $\mathcal{T}_n^j$ and in $\mathcal{T}_n$, respectively. Set $A_n:=\bigcap_{1\leq j\leq n^3}A_n^j$. It follows from the union bound that
\begin{equation} \label{eq400}
\mathbb{P}(A_n^c) \leq n^3 \mathbb{P}\left((A_n^j)^c\right). 
\end{equation}
Let $S_{n,k}$ and $\mathcal{B}_{n,k}$ be, respectively, as in \eqref{eqsnm} and \eqref{eqbnm} with the replacement of $R_n$ therein by $r_n+\sqrt{d}$. Set $b_n:=\mathbb{P}(\mathcal{B}_{n,k})$. The proof of \cite[Lemma 3.3]{S1993} shows that $b_n\geq \frac{1}{2}\exp\left[\omega_d (r_n+\sqrt{d})^d \ln p\right]$ for all large $n$, and via Chebyshev's inequality and a variance calculation that 
\begin{equation} \label{eq401}
\mathbb{P}\left[\sum_k \mathbbm{1}_{\mathcal{B}_{n,k}}<2^{(n+1)d}b_n/2\right] \leq \frac{4}{2^{(n+1)d}b_n}+32\beta e^{-\gamma\left(\frac{2 (r_n+\sqrt{d})}{\sqrt{d}}-2\right)} ,
\end{equation}
where $\beta>1$ and $\gamma>0$ are suitable positive constants. It is clear that 
$$ E_n^c   \subseteq    \left\{\sum_k \mathbbm{1}_{\mathcal{B}_{n,k}}=0\right\}     \subseteq\left\{\sum_k \mathbbm{1}_{\mathcal{B}_{n,k}}<2^{(n+1)d}b_n/2\right\}.   $$ 
Moreover, the proof of \cite[Lemma 3.3]{S1993} does not make use of the center of $\mathcal{T}_n$ but only uses the sidelength of $\mathcal{T}_n$, and therefore we may use the right-hand side of \eqref{eq401} as an upper bound for $\mathbb{P}((A_n^j)^c)$ for each $j$. It then follows from \eqref{eq400} and \eqref{eq401} that
\begin{equation} \label{eq402}
\sum_{n=1}^\infty \mathbb{P}(A_n^c) \leq \sum_{n=1}^\infty \left[\frac{4 n^3}{2^{(n+1)d}b_n}+32 n^3 \beta e^{-\gamma\left(\frac{2 (r_n+\sqrt{d})}{\sqrt{d}}-2\right)} \right]. 
\end{equation}
In view of $b_n\geq \frac{1}{2}\exp\left[\omega_d (r_n+\sqrt{d})^d \ln p\right]$, and using \eqref{eqro} and the definition of $r_n$ from \eqref{eqrn2}, one can show that for all large $n$ the first term inside the summation in \eqref{eq402} is bounded from above by 
$$ \frac{8}{2^d} \frac{n^3}{2^{nd/2}} . $$
Moreover, for all large $n$ the second term inside the summation in \eqref{eq402} is bounded from above by
$$ 32\beta n^3 \exp\left[-\left(\gamma  R_0 (\log 2)^{1/d}/\sqrt{d}\right)n^{1/d}\right]  . $$
Hence, the series in \eqref{eq402} is summable, and an application of the Borel-Cantelli lemma completes the proof.  
\end{proof}

\bigskip

\textbf{\underline{Part 1}: Many large clearings.} Let $\rho:\mathbb{N}_{\geq 3}\to\mathbb{R}_+$ be defined by 
\begin{equation} \label{eqrhot}
\rho(m)=k_2\frac{m}{(\log m)^{2/d}} , \quad k_2>0 . 
\end{equation}
Observe that $\rho$ is increasing and that $\lim_{m\to\infty}\rho(m)=\infty$. Now choose $n=n(m)$ such that for large enough $m$, 
\begin{equation} \nonumber
2^{n(m)}\log 2^{n(m)} \leq \rho(m) < 2^{n(m)+1}\log 2^{n(m)+1} .
\end{equation} 
This implies that for any fixed $\alpha>1$, for all large $n$,
\begin{equation} \label{eq404}
\log \rho(m) - \alpha\log\log \rho(m) \leq \log 2^{n(m)} .
\end{equation}
Then, it follows from \eqref{eq404} and the choice of $\rho(m)$ in \eqref{eqrhot} that for all large $n$,
$$ n^2 \geq \frac{1}{2(\log 2)^2}(\log m)^2 ,$$
where we have dropped the $m$-dependence of $n$ in notation. This implies that for any fixed constant $c>0$ there exists $n_0=n_0(c)\geq 1$ such that for all $n\geq n_0$,
\begin{equation} \label{eqn3}
n^3 \geq c(\log m)^2 .  
\end{equation} 
Set 
\begin{equation} \nonumber
\widehat{\rho}(m) := \left\lfloor \frac{\rho(m)}{2\sqrt{d}} \right\rfloor. 
\end{equation}   
Now let $A>0$, and for large $m\in\mathbb{N}$, consider the simple cubic packing of $[-Am,Am]^d$ with balls of radius $\widehat{\rho}(m)$. In detail, for $j\in \left[ -\left\lceil \frac{Am}{2\widehat{\rho}(m)} \right\rceil , \left\lceil \frac{Am}{2\widehat{\rho}(m)}\right\rceil -1\right]^d \cap \mathbb{Z}^d =: \textbf{S}(m)$, introduce the cubes
$$  C_j := \left\{ z\in\mathbb{R}^d : j_i\, 2\widehat{\rho}(m) \leq z_i \leq (j_i+1)2\widehat{\rho}(m) \right\}   .$$
Observe that each $C_j$ is centered at a point in $\mathbb{Z}^d$ and that 
$$ \bigcup_{j\in\textbf{S}} C_j \supseteq  [-Am,Am]^d .   $$
For $j\in\textbf{S}(m)$, let $B_m^j=B(j,\widehat{\rho}(m))$ be the ball inscribed in $C_j$. From elementary geometry, one can show that
\begin{equation} \label{eq406}
\forall \: x \: \in [-Am,Am]^d \cap \mathbb{Z}^d, \quad\:\: \min_{j\in\textbf{S}(m)} \max_{z\in B_m^j \cap \mathbb{Z}^d} |x-z| < \rho(m) .
\end{equation}
On the other hand, for all large $n$,
\begin{equation} \label{eqnt}
 |\textbf{S}(m)| = \left(2\left\lceil \frac{Am}{2\widehat{\rho}(m)} \right\rceil\right)^d \leq \left(\frac{3A\sqrt{d}}{k_2}\right)^d(\log m)^2 \leq n^3,
\end{equation}
where the last inequality follows from \eqref{eqn3} upon setting $c=(3A\sqrt{d}/k_2)^d$ therein. 

Next, recall $r_n=(2R_0/3)(\log 2^n)^{1/d}$ from \eqref{eqrn2}. Using \eqref{eq404}, we have for all large $m$, 
\begin{equation} \label{eq405}
(\log 2^n)^{1/d} \geq (\log \rho(m))^{1/d}\left[1-\frac{\alpha\log\log\rho(m)}{\log\rho(m)}\right]^{1/d} \geq (3/4)(\log \rho(m))^{1/d}.
\end{equation}
On the other hand, by definition of $\mathcal{T}_n$ in \eqref{eqtn} and since $d\geq 2$, it follows from  $\rho(m)\geq 2^{n(m)}\log 2^{n(m)}$ that for all large $m$,
\begin{equation} \label{eq4050}
[-\widehat{\rho}(m),\widehat{\rho}(m)]^d \supseteq \mathcal{T}_n.  
\end{equation}
Then, by Lemma~\ref{lemma1}, upon setting $x_j=z_j$ for $x_j$ therein, and using \eqref{eq406}-\eqref{eq4050}, we reach the following conclusion. On a set of full $\mathbb{P}$-measure, say $\Omega_3$, there exists $m_3=m_3(\omega)>0$ such that for all $m\geq m_3$, any cube of sidelength $2\rho(m)$ centered at a site within $[-Am,Am]^d$ contains an accessible clearing of radius $(R_0/2)(\log \rho(m))^{1/d}$. 

\bigskip

\textbf{\underline{Part 2}: One huge clearing.} Choose $n=n(m)$ such that for large enough $m$,
\begin{equation} \nonumber
2^{n}\log 2^{n} \leq Am < 2^{n+1}\log 2^{n+1} .
\end{equation}
After some algebra, one can show from $Am < 2^{n+1}\log 2^{n+1}$ and since $d\geq 2$ that for each fixed $\delta>0$,
\begin{equation} \nonumber
(\log 2^n)^{1/d} \geq (\log m)^{1/d}-\delta
\end{equation} 
for all large $m$. Then, by Lemma A, we reach the following conclusion. On a set of full $\mathbb{P}$-measure, say $\Omega_4$, there exists $m_4=m_4(\omega)$ such that for all $m\geq m_4$, $[-Am,Am]^d$ contains an accessible clearing of radius $R_0(\log m)^{1/d}-2\sqrt{d}$. 

Summarizing our findings in this section, we have $\mathbb{P}(\Omega_3)=1=\mathbb{P}(\Omega_4)$, where 
\begin{align}
\Omega_3&=\Big\{\omega\in\Omega:\exists\:n_3=n_3(\omega) \:\: \forall\:n\geq n_3, \: \forall\:x\in[-An,An]^d\cap\mathbb{Z}^d \:\:\exists\: y\in x+[-\rho(n),\rho(n)]^d \nonumber \\
& \quad \quad \text{such that} \:\: B\left(y,(R_0/2)(\log \rho(n))^{1/d}\right)\cap\mathbb{Z}^d\subseteq \mathcal{C} \Big\}, \label{eqomega3} \\
\Omega_4&=\Big\{\omega\in\Omega:\exists\:n_4=n_4(\omega) \:\: \forall\:n\geq n_4, \:\:\exists\: y\in \mathbb{R}^d\:\:\text{such that} \nonumber \\
& \quad \quad B\left(y,R_0(\log n))^{1/d}-2\sqrt{d}\right)\cap\mathbb{Z}^d\subseteq [-An,An]^d\cap \mathcal{C} \Big\} . \label{eqomega4}
\end{align}
We refer the reader to Figure~\ref{figure1} for an illustration of an environment in $\Omega_3\cap\Omega_4$. Finally, recall the definitions of $\Omega_1$ from \eqref{eqomega1} and $\Omega_2$ from \eqref{eqomega2}, and set 
\begin{equation} \label{eqomega0}
\Omega_0 = \bigcap_{1\leq i\leq 4}\Omega_i  . 
\end{equation}
Then, $\mathbb{P}(\Omega_0)=1$ since $\mathbb{P}(\Omega_i)=1$ for each $1\leq i\leq 4$. We will use $\Omega_0$ as our a.s.-environment such that the SLLN for the mass of the BRW among Bernoulli traps in $\mathbb{Z}^d$ holds on $\Omega_0 \cap \{\mathbf{0}\in\mathcal{C}\}$.

\begin{figure}[ht]
\begin{center}

\begin{tikzpicture}[scale=1.1]

\draw[step=2.5cm,gray,line width=0.01mm] (-0.3,-0.3) grid (10.3,10.3);
\draw[<->, very thin] (10.33,-0.3) -- (10.33,5);
\node[] at (10.6,2.5) {\tiny $A n$};
\draw[<->, very thin] (10.2,7.5) -- (10.2,10);
\node[] at (10.6,8.7) {\tiny $\widehat\rho(n)$};
\draw[<->, very thin] (5,-0.3) -- (10.3,-0.3);
\node[] at (7.5,-0.5) {\tiny $A n$};
\draw[<->, very thin] (7.5,10.2) -- (10,10.2);
\node[] at (8.75,10.4) {\tiny $\widehat\rho(n)$};
\node at (5,5) [thin]{.};
\node[] at (5.15,4.85) {\tiny $0$};

\filldraw[fill=white, draw=gray] (0,0) circle (0.04cm);
\filldraw[fill=white, draw=gray] (0.5,0) circle (0.04cm);
\filldraw[fill=red, draw=red] (1,0) circle (0.04cm);
\filldraw[fill=white, draw=gray] (1.5,0) circle (0.04cm);
\filldraw[fill=red, draw=red] (2,0) circle (0.04cm);
\filldraw[fill=white, draw=gray] (2.5,0) circle (0.04cm);
\filldraw[fill=red, draw=red] (3,0) circle (0.04cm);
\filldraw[fill=white, draw=gray] (3.5,0) circle (0.04cm);
\filldraw[fill=white, draw=gray] (4,0) circle (0.04cm);
\filldraw[fill=red, draw=red] (4.5,0) circle (0.04cm);
\filldraw[fill=white, draw=gray] (5,0) circle (0.04cm);
\filldraw[fill=red, draw=red] (5.5,0) circle (0.04cm);
\filldraw[fill=red, draw=red] (6,0) circle (0.04cm);
\filldraw[fill=red, draw=red] (6.5,0) circle (0.04cm);
\filldraw[fill=white, draw=gray] (7,0) circle (0.04cm);
\filldraw[fill=white, draw=gray] (7.5,0) circle (0.04cm);
\filldraw[fill=red, draw=red] (8,0) circle (0.04cm);
\filldraw[fill=red, draw=red] (8.5,0) circle (0.04cm);
\filldraw[fill=red, draw=red] (9,0) circle (0.04cm);
\filldraw[fill=white, draw=gray] (9.5,0) circle (0.04cm);
\filldraw[fill=white, draw=gray] (10,0) circle (0.04cm);

\filldraw[fill=red, draw=red] (0,0.5) circle (0.04cm);
\filldraw[fill=white, draw=gray] (0.5,0.5) circle (0.04cm);
\filldraw[fill=red, draw=red] (1,0.5) circle (0.04cm);
\filldraw[fill=white, draw=gray] (1.5,0.5) circle (0.04cm);
\filldraw[fill=white, draw=gray] (2,0.5) circle (0.04cm);
\filldraw[fill=white, draw=gray] (2.5,0.5) circle (0.04cm);
\filldraw[fill=red, draw=red] (3,0.5) circle (0.04cm);
\filldraw[fill=red, draw=red] (3.5,0.5) circle (0.04cm);
\filldraw[fill=white, draw=gray] (4,0.5) circle (0.04cm);
\filldraw[fill=red, draw=red] (4.5,0.5) circle (0.04cm);
\filldraw[fill=red, draw=red] (5,0.5) circle (0.04cm);
\filldraw[fill=red, draw=red] (5.5,0.5) circle (0.04cm);
\filldraw[fill=red, draw=red] (6,0.5) circle (0.04cm);
\filldraw[fill=white, draw=gray] (6.5,0.5) circle (0.04cm);
\filldraw[fill=white, draw=gray] (7,0.5) circle (0.04cm);
\filldraw[fill=white, draw=gray] (7.5,0.5) circle (0.04cm);
\filldraw[fill=red, draw=red] (8,0.5) circle (0.04cm);
\filldraw[fill=white, draw=gray] (8.5,0.5) circle (0.04cm);
\filldraw[fill=red, draw=red] (9,0.5) circle (0.04cm);
\filldraw[fill=white, draw=gray] (9.5,0.5) circle (0.04cm);
\filldraw[fill=white, draw=gray] (10,0.5) circle (0.04cm);

\filldraw[fill=red, draw=red] (0,1) circle (0.04cm);
\filldraw[fill=red, draw=red] (0.5,1) circle (0.04cm);
\filldraw[fill=white, draw=gray] (1,1) circle (0.04cm);
\filldraw[fill=white, draw=gray] (1.5,1) circle (0.04cm);
\filldraw[fill=white, draw=gray] (2,1) circle (0.04cm);
\filldraw[fill=red, draw=red] (2.5,1) circle (0.04cm);
\filldraw[fill=red, draw=red] (3,1) circle (0.04cm);
\filldraw[fill=red, draw=red] (3.5,1) circle (0.04cm);
\filldraw[fill=white, draw=gray] (4,1) circle (0.04cm);
\filldraw[fill=white, draw=gray] (4.5,1) circle (0.04cm);
\filldraw[fill=white, draw=gray] (5,1) circle (0.04cm);
\filldraw[fill=white, draw=gray] (5.5,1) circle (0.04cm);
\filldraw[fill=white, draw=gray] (6,1) circle (0.04cm);
\filldraw[fill=white, draw=gray] (6.5,1) circle (0.04cm);
\filldraw[fill=white, draw=gray] (7,1) circle (0.04cm);
\filldraw[fill=red, draw=red] (7.5,1) circle (0.04cm);
\filldraw[fill=white, draw=gray] (8,1) circle (0.04cm);
\filldraw[fill=white, draw=gray] (8.5,1) circle (0.04cm);
\filldraw[fill=red, draw=red] (9,1) circle (0.04cm);
\filldraw[fill=red, draw=red] (9.5,1) circle (0.04cm);
\filldraw[fill=red, draw=red] (10,1) circle (0.04cm);

\filldraw[fill=red, draw=red] (0,1.5) circle (0.04cm);
\filldraw[fill=white, draw=gray] (0.5,1.5) circle (0.04cm);
\filldraw[fill=white, draw=gray] (1,1.5) circle (0.04cm);
\filldraw[fill=white, draw=gray] (1.5,1.5) circle (0.04cm);
\filldraw[fill=white, draw=gray] (2,1.5) circle (0.04cm);
\filldraw[fill=white, draw=gray] (2.5,1.5) circle (0.04cm);
\filldraw[fill=red, draw=red] (3,1.5) circle (0.04cm);
\filldraw[fill=red, draw=red] (3.5,1.5) circle (0.04cm);
\filldraw[fill=white, draw=gray] (4,1.5) circle (0.04cm);
\filldraw[fill=red, draw=red] (4.5,1.5) circle (0.04cm);
\filldraw[fill=white, draw=gray] (5,1.5) circle (0.04cm);
\filldraw[fill=red, draw=red] (5.5,1.5) circle (0.04cm);
\filldraw[fill=red, draw=red] (6,1.5) circle (0.04cm);
\filldraw[fill=white, draw=gray] (6.5,1.5) circle (0.04cm);
\filldraw[fill=white, draw=gray] (7,1.5) circle (0.04cm);
\filldraw[fill=red, draw=red] (7.5,1.5) circle (0.04cm);
\filldraw[fill=red, draw=red] (8,1.5) circle (0.04cm);
\filldraw[fill=red, draw=red] (8.5,1.5) circle (0.04cm);
\filldraw[fill=white, draw=gray] (9,1.5) circle (0.04cm);
\filldraw[fill=white, draw=gray] (9.5,1.5) circle (0.04cm);
\filldraw[fill=red, draw=red] (10,1.5) circle (0.04cm);

\filldraw[fill=white, draw=gray] (0,2) circle (0.04cm);
\filldraw[fill=white, draw=gray] (0.5,2) circle (0.04cm);
\filldraw[fill=white, draw=gray] (1,2) circle (0.04cm);
\filldraw[fill=white, draw=gray] (1.5,2) circle (0.04cm);
\filldraw[fill=white, draw=gray] (2,2) circle (0.04cm);
\filldraw[fill=white, draw=gray] (2.5,2) circle (0.04cm);
\filldraw[fill=red, draw=red] (3,2) circle (0.04cm);
\filldraw[fill=white, draw=gray] (3.5,2) circle (0.04cm);
\filldraw[fill=white, draw=gray] (4,2) circle (0.04cm);
\filldraw[fill=red, draw=red] (4.5,2) circle (0.04cm);
\filldraw[fill=red, draw=red] (5,2) circle (0.04cm);
\filldraw[fill=red, draw=red] (5.5,2) circle (0.04cm);
\filldraw[fill=white, draw=gray] (6,2) circle (0.04cm);
\filldraw[fill=red, draw=red] (6.5,2) circle (0.04cm);
\filldraw[fill=red, draw=red] (7,2) circle (0.04cm);
\filldraw[fill=white, draw=gray] (7.5,2) circle (0.04cm);
\filldraw[fill=white, draw=gray] (8,2) circle (0.04cm);
\filldraw[fill=red, draw=red] (8.5,2) circle (0.04cm);
\filldraw[fill=red, draw=red] (9,2) circle (0.04cm);
\filldraw[fill=white, draw=gray] (9.5,2) circle (0.04cm);
\filldraw[fill=white, draw=gray] (10,2) circle (0.04cm);

\filldraw[fill=white, draw=gray] (0,2.5) circle (0.04cm);
\filldraw[fill=white, draw=gray] (0.5,2.5) circle (0.04cm);
\filldraw[fill=white, draw=gray] (1,2.5) circle (0.04cm);
\filldraw[fill=white, draw=gray] (1.5,2.5) circle (0.04cm);
\filldraw[fill=white, draw=gray] (2,2.5) circle (0.04cm);
\filldraw[fill=white, draw=gray] (2.5,2.5) circle (0.04cm);
\filldraw[fill=red, draw=red] (3,2.5) circle (0.04cm);
\filldraw[fill=white, draw=gray] (3.5,2.5) circle (0.04cm);
\filldraw[fill=white, draw=gray] (4,2.5) circle (0.04cm);
\filldraw[fill=red, draw=red] (4.5,2.5) circle (0.04cm);
\filldraw[fill=red, draw=red] (5,2.5) circle (0.04cm);
\filldraw[fill=white, draw=gray] (5.5,2.5) circle (0.04cm);
\filldraw[fill=red, draw=red] (6,2.5) circle (0.04cm);
\filldraw[fill=red, draw=red] (6.5,2.5) circle (0.04cm);
\filldraw[fill=white, draw=gray] (7,2.5) circle (0.04cm);
\filldraw[fill=white, draw=gray] (7.5,2.5) circle (0.04cm);
\filldraw[fill=red, draw=red] (8,2.5) circle (0.04cm);
\filldraw[fill=red, draw=red] (8.5,2.5) circle (0.04cm);
\filldraw[fill=red, draw=red] (9,2.5) circle (0.04cm);
\filldraw[fill=white, draw=gray] (9.5,2.5) circle (0.04cm);
\filldraw[fill=white, draw=gray] (10,2.5) circle (0.04cm);

\filldraw[fill=red, draw=red] (0,3) circle (0.04cm);
\filldraw[fill=white, draw=gray] (0.5,3) circle (0.04cm);
\filldraw[fill=white, draw=gray] (1,3) circle (0.04cm);
\filldraw[fill=white, draw=gray] (1.5,3) circle (0.04cm);
\filldraw[fill=white, draw=gray] (2,3) circle (0.04cm);
\filldraw[fill=white, draw=gray] (2.5,3) circle (0.04cm);
\filldraw[fill=white, draw=gray] (3,3) circle (0.04cm);
\filldraw[fill=white, draw=gray] (3.5,3) circle (0.04cm);
\filldraw[fill=white, draw=gray] (4,3) circle (0.04cm);
\filldraw[fill=red, draw=red] (4.5,3) circle (0.04cm);
\filldraw[fill=red, draw=red] (5,3) circle (0.04cm);
\filldraw[fill=white, draw=gray] (5.5,3) circle (0.04cm);
\filldraw[fill=red, draw=red] (6,3) circle (0.04cm);
\filldraw[fill=red, draw=red] (6.5,3) circle (0.04cm);
\filldraw[fill=white, draw=gray] (7,3) circle (0.04cm);
\filldraw[fill=white, draw=gray] (7.5,3) circle (0.04cm);
\filldraw[fill=white, draw=gray] (8,3) circle (0.04cm);
\filldraw[fill=red, draw=red] (8.5,3) circle (0.04cm);
\filldraw[fill=white, draw=gray] (9,3) circle (0.04cm);
\filldraw[fill=white, draw=gray] (9.5,3) circle (0.04cm);
\filldraw[fill=white, draw=gray] (10,3) circle (0.04cm);

\filldraw[fill=red, draw=red] (0,3.5) circle (0.04cm);
\filldraw[fill=red, draw=red] (0.5,3.5) circle (0.04cm);
\filldraw[fill=white, draw=gray] (1,3.5) circle (0.04cm);
\filldraw[fill=white, draw=gray] (1.5,3.5) circle (0.04cm);
\filldraw[fill=white, draw=gray] (2,3.5) circle (0.04cm);
\filldraw[fill=white, draw=gray] (2.5,3.5) circle (0.04cm);
\filldraw[fill=red, draw=red] (3,3.5) circle (0.04cm);
\filldraw[fill=white, draw=gray] (3.5,3.5) circle (0.04cm);
\filldraw[fill=white, draw=gray] (4,3.5) circle (0.04cm);
\filldraw[fill=red, draw=red] (4.5,3.5) circle (0.04cm);
\filldraw[fill=red, draw=red] (5,3.5) circle (0.04cm);
\filldraw[fill=white, draw=gray] (5.5,3.5) circle (0.04cm);
\filldraw[fill=red, draw=red] (6,3.5) circle (0.04cm);
\filldraw[fill=red, draw=red] (6.5,3.5) circle (0.04cm);
\filldraw[fill=white, draw=gray] (7,3.5) circle (0.04cm);
\filldraw[fill=white, draw=gray] (7.5,3.5) circle (0.04cm);
\filldraw[fill=red, draw=red] (8,3.5) circle (0.04cm);
\filldraw[fill=red, draw=red] (8.5,3.5) circle (0.04cm);
\filldraw[fill=white, draw=gray] (9,3.5) circle (0.04cm);
\filldraw[fill=white, draw=gray] (9.5,3.5) circle (0.04cm);
\filldraw[fill=red, draw=red] (10,3.5) circle (0.04cm);

\filldraw[fill=red, draw=red] (0,4) circle (0.04cm);
\filldraw[fill=white, draw=gray] (0.5,4) circle (0.04cm);
\filldraw[fill=white, draw=gray] (1,4) circle (0.04cm);
\filldraw[fill=white, draw=gray] (1.5,4) circle (0.04cm);
\filldraw[fill=red, draw=red] (2,4) circle (0.04cm);
\filldraw[fill=white, draw=gray] (2.5,4) circle (0.04cm);
\filldraw[fill=white, draw=gray] (3,4) circle (0.04cm);
\filldraw[fill=red, draw=red] (3.5,4) circle (0.04cm);
\filldraw[fill=white, draw=gray] (4,4) circle (0.04cm);
\filldraw[fill=white, draw=gray] (4.5,4) circle (0.04cm);
\filldraw[fill=red, draw=red] (5,4) circle (0.04cm);
\filldraw[fill=white, draw=gray] (5.5,4) circle (0.04cm);
\filldraw[fill=red, draw=red] (6,4) circle (0.04cm);
\filldraw[fill=red, draw=red] (6.5,4) circle (0.04cm);
\filldraw[fill=white, draw=gray] (7,4) circle (0.04cm);
\filldraw[fill=red, draw=red] (7.5,4) circle (0.04cm);
\filldraw[fill=white, draw=gray] (8,4) circle (0.04cm);
\filldraw[fill=red, draw=red] (8.5,4) circle (0.04cm);
\filldraw[fill=red, draw=red] (9,4) circle (0.04cm);
\filldraw[fill=white, draw=gray] (9.5,4) circle (0.04cm);
\filldraw[fill=white, draw=gray] (10,4) circle (0.04cm);

\filldraw[fill=red, draw=red] (0,4.5) circle (0.04cm);
\filldraw[fill=red, draw=red] (0.5,4.5) circle (0.04cm);
\filldraw[fill=red, draw=red] (1,4.5) circle (0.04cm);
\filldraw[fill=white, draw=gray] (1.5,4.5) circle (0.04cm);
\filldraw[fill=red, draw=red] (2,4.5) circle (0.04cm);
\filldraw[fill=white, draw=gray] (2.5,4.5) circle (0.04cm);
\filldraw[fill=red, draw=red] (3,4.5) circle (0.04cm);
\filldraw[fill=white, draw=gray] (3.5,4.5) circle (0.04cm);
\filldraw[fill=red, draw=red] (4,4.5) circle (0.04cm);
\filldraw[fill=white, draw=gray] (4.5,4.5) circle (0.04cm);
\filldraw[fill=red, draw=red] (5,4.5) circle (0.04cm);
\filldraw[fill=white, draw=gray] (5.5,4.5) circle (0.04cm);
\filldraw[fill=red, draw=red] (6,4.5) circle (0.04cm);
\filldraw[fill=white, draw=gray] (6.5,4.5) circle (0.04cm);
\filldraw[fill=red, draw=red] (7,4.5) circle (0.04cm);
\filldraw[fill=white, draw=gray] (7.5,4.5) circle (0.04cm);
\filldraw[fill=red, draw=red] (8,4.5) circle (0.04cm);
\filldraw[fill=white, draw=gray] (8.5,4.5) circle (0.04cm);
\filldraw[fill=white, draw=gray] (9,4.5) circle (0.04cm);
\filldraw[fill=white, draw=gray] (9.5,4.5) circle (0.04cm);
\filldraw[fill=red, draw=red] (10,4.5) circle (0.04cm);

\filldraw[fill=red, draw=red] (0,5) circle (0.04cm);
\filldraw[fill=red, draw=red] (0.5,5) circle (0.04cm);
\filldraw[fill=white, draw=gray] (1,5) circle (0.04cm);
\filldraw[fill=white, draw=gray] (1.5,5) circle (0.04cm);
\filldraw[fill=red, draw=red] (2,5) circle (0.04cm);
\filldraw[fill=white, draw=gray] (2.5,5) circle (0.04cm);
\filldraw[fill=white, draw=gray] (3,5) circle (0.04cm);
\filldraw[fill=red, draw=red] (3.5,5) circle (0.04cm);
\filldraw[fill=red, draw=red] (4,5) circle (0.04cm);
\filldraw[fill=white, draw=gray] (4.5,5) circle (0.04cm);
\filldraw[fill=white, draw=gray] (5,5) circle (0.04cm);
\filldraw[fill=white, draw=gray] (5.5,5) circle (0.04cm);
\filldraw[fill=red, draw=red] (6,5) circle (0.04cm);
\filldraw[fill=red, draw=red] (6.5,5) circle (0.04cm);
\filldraw[fill=white, draw=gray] (7,5) circle (0.04cm);
\filldraw[fill=red, draw=red] (7.5,5) circle (0.04cm);
\filldraw[fill=white, draw=gray] (8,5) circle (0.04cm);
\filldraw[fill=red, draw=red] (8.5,5) circle (0.04cm);
\filldraw[fill=white, draw=gray] (9,5) circle (0.04cm);
\filldraw[fill=white, draw=gray] (9.5,5) circle (0.04cm);
\filldraw[fill=white, draw=gray] (10,5) circle (0.04cm);

\filldraw[fill=red, draw=red] (0,5.5) circle (0.04cm);
\filldraw[fill=white, draw=gray] (0.5,5.5) circle (0.04cm);
\filldraw[fill=red, draw=red] (1,5.5) circle (0.04cm);
\filldraw[fill=white, draw=gray] (1.5,5.5) circle (0.04cm);
\filldraw[fill=red, draw=red] (2,5.5) circle (0.04cm);
\filldraw[fill=white, draw=gray] (2.5,5.5) circle (0.04cm);
\filldraw[fill=red, draw=red] (3,5.5) circle (0.04cm);
\filldraw[fill=white, draw=gray] (3.5,5.5) circle (0.04cm);
\filldraw[fill=red, draw=red] (4,5.5) circle (0.04cm);
\filldraw[fill=white, draw=gray] (4.5,5.5) circle (0.04cm);
\filldraw[fill=red, draw=red] (5,5.5) circle (0.04cm);
\filldraw[fill=white, draw=gray] (5.5,5.5) circle (0.04cm);
\filldraw[fill=red, draw=red] (6,5.5) circle (0.04cm);
\filldraw[fill=white, draw=gray] (6.5,5.5) circle (0.04cm);
\filldraw[fill=white, draw=gray] (7,5.5) circle (0.04cm);
\filldraw[fill=white, draw=gray] (7.5,5.5) circle (0.04cm);
\filldraw[fill=red, draw=red] (8,5.5) circle (0.04cm);
\filldraw[fill=white, draw=gray] (8.5,5.5) circle (0.04cm);
\filldraw[fill=white, draw=gray] (9,5.5) circle (0.04cm);
\filldraw[fill=white, draw=gray] (9.5,5.5) circle (0.04cm);
\filldraw[fill=white, draw=gray] (10,5.5) circle (0.04cm);

\filldraw[fill=white, draw=gray] (0,6) circle (0.04cm);
\filldraw[fill=red, draw=red] (0.5,6) circle (0.04cm);
\filldraw[fill=white, draw=gray] (1,6) circle (0.04cm);
\filldraw[fill=white, draw=gray] (1.5,6) circle (0.04cm);
\filldraw[fill=red, draw=red] (2,6) circle (0.04cm);
\filldraw[fill=white, draw=gray] (2.5,6) circle (0.04cm);
\filldraw[fill=white, draw=gray] (3,6) circle (0.04cm);
\filldraw[fill=red, draw=red] (3.5,6) circle (0.04cm);
\filldraw[fill=red, draw=red] (4,6) circle (0.04cm);
\filldraw[fill=white, draw=gray] (4.5,6) circle (0.04cm);
\filldraw[fill=red, draw=red] (5,6) circle (0.04cm);
\filldraw[fill=white, draw=gray] (5.5,6) circle (0.04cm);
\filldraw[fill=red, draw=red] (6,6) circle (0.04cm);
\filldraw[fill=red, draw=red] (6.5,6) circle (0.04cm);
\filldraw[fill=white, draw=gray] (7,6) circle (0.04cm);
\filldraw[fill=red, draw=red] (7.5,6) circle (0.04cm);
\filldraw[fill=red, draw=red] (8,6) circle (0.04cm);
\filldraw[fill=red, draw=red] (8.5,6) circle (0.04cm);
\filldraw[fill=white, draw=gray] (9,6) circle (0.04cm);
\filldraw[fill=white, draw=gray] (9.5,6) circle (0.04cm);
\filldraw[fill=white, draw=gray] (10,6) circle (0.04cm);

\filldraw[fill=white, draw=gray] (0,6.5) circle (0.04cm);
\filldraw[fill=red, draw=red] (0.5,6.5) circle (0.04cm);
\filldraw[fill=white, draw=gray] (1,6.5) circle (0.04cm);
\filldraw[fill=red, draw=red] (1.5,6.5) circle (0.04cm);
\filldraw[fill=red, draw=red] (2,6.5) circle (0.04cm);
\filldraw[fill=red, draw=red] (2.5,6.5) circle (0.04cm);
\filldraw[fill=white, draw=gray] (3,6.5) circle (0.04cm);
\filldraw[fill=red, draw=red] (3.5,6.5) circle (0.04cm);
\filldraw[fill=white, draw=gray] (4,6.5) circle (0.04cm);
\filldraw[fill=white, draw=gray] (4.5,6.5) circle (0.04cm);
\filldraw[fill=white, draw=gray] (5,6.5) circle (0.04cm);
\filldraw[fill=red, draw=red] (5.5,6.5) circle (0.04cm);
\filldraw[fill=white, draw=gray] (6,6.5) circle (0.04cm);
\filldraw[fill=white, draw=gray] (6.5,6.5) circle (0.04cm);
\filldraw[fill=red, draw=red] (7,6.5) circle (0.04cm);
\filldraw[fill=red, draw=red] (7.5,6.5) circle (0.04cm);
\filldraw[fill=white, draw=gray] (8,6.5) circle (0.04cm);
\filldraw[fill=red, draw=red] (8.5,6.5) circle (0.04cm);
\filldraw[fill=white, draw=gray] (9,6.5) circle (0.04cm);
\filldraw[fill=red, draw=red] (9.5,6.5) circle (0.04cm);
\filldraw[fill=red, draw=red] (10,6.5) circle (0.04cm);

\filldraw[fill=red, draw=red] (0,7) circle (0.04cm);
\filldraw[fill=red, draw=red] (0.5,7) circle (0.04cm);
\filldraw[fill=white, draw=gray] (1,7) circle (0.04cm);
\filldraw[fill=white, draw=gray] (1.5,7) circle (0.04cm);
\filldraw[fill=red, draw=red] (2,7) circle (0.04cm);
\filldraw[fill=white, draw=gray] (2.5,7) circle (0.04cm);
\filldraw[fill=white, draw=gray] (3,7) circle (0.04cm);
\filldraw[fill=white, draw=gray] (3.5,7) circle (0.04cm);
\filldraw[fill=red, draw=red] (4,7) circle (0.04cm);
\filldraw[fill=red, draw=red] (4.5,7) circle (0.04cm);
\filldraw[fill=white, draw=gray] (5,7) circle (0.04cm);
\filldraw[fill=white, draw=gray] (5.5,7) circle (0.04cm);
\filldraw[fill=white, draw=gray] (6,7) circle (0.04cm);
\filldraw[fill=white, draw=gray] (6.5,7) circle (0.04cm);
\filldraw[fill=white, draw=gray] (7,7) circle (0.04cm);
\filldraw[fill=red, draw=red] (7.5,7) circle (0.04cm);
\filldraw[fill=red, draw=red] (8,7) circle (0.04cm);
\filldraw[fill=red, draw=red] (8.5,7) circle (0.04cm);
\filldraw[fill=white, draw=gray] (9,7) circle (0.04cm);
\filldraw[fill=red, draw=red] (9.5,7) circle (0.04cm);
\filldraw[fill=red, draw=red] (10,7) circle (0.04cm);

\filldraw[fill=white, draw=gray] (0,7.5) circle (0.04cm);
\filldraw[fill=red, draw=red] (0.5,7.5) circle (0.04cm);
\filldraw[fill=red, draw=red] (1,7.5) circle (0.04cm);
\filldraw[fill=white, draw=gray] (1.5,7.5) circle (0.04cm);
\filldraw[fill=red, draw=red] (2,7.5) circle (0.04cm);
\filldraw[fill=red, draw=red] (2.5,7.5) circle (0.04cm);
\filldraw[fill=white, draw=gray] (3,7.5) circle (0.04cm);
\filldraw[fill=red, draw=red] (3.5,7.5) circle (0.04cm);
\filldraw[fill=white, draw=gray] (4,7.5) circle (0.04cm);
\filldraw[fill=white, draw=gray] (4.5,7.5) circle (0.04cm);
\filldraw[fill=white, draw=gray] (5,7.5) circle (0.04cm);
\filldraw[fill=red, draw=red] (5.5,7.5) circle (0.04cm);
\filldraw[fill=red, draw=red] (6,7.5) circle (0.04cm);
\filldraw[fill=white, draw=gray] (6.5,7.5) circle (0.04cm);
\filldraw[fill=white, draw=gray] (7,7.5) circle (0.04cm);
\filldraw[fill=white, draw=gray] (7.5,7.5) circle (0.04cm);
\filldraw[fill=white, draw=gray] (8,7.5) circle (0.04cm);
\filldraw[fill=red, draw=red] (8.5,7.5) circle (0.04cm);
\filldraw[fill=white, draw=gray] (9,7.5) circle (0.04cm);
\filldraw[fill=red, draw=red] (9.5,7.5) circle (0.04cm);
\filldraw[fill=white, draw=gray] (10,7.5) circle (0.04cm);

\filldraw[fill=white, draw=gray] (0,8) circle (0.04cm);
\filldraw[fill=red, draw=red] (0.5,8) circle (0.04cm);
\filldraw[fill=white, draw=gray] (1,8) circle (0.04cm);
\filldraw[fill=white, draw=gray] (1.5,8) circle (0.04cm);
\filldraw[fill=red, draw=red] (2,8) circle (0.04cm);
\filldraw[fill=white, draw=gray] (2.5,8) circle (0.04cm);
\filldraw[fill=white, draw=gray] (3,8) circle (0.04cm);
\filldraw[fill=white, draw=gray] (3.5,8) circle (0.04cm);
\filldraw[fill=red, draw=red] (4,8) circle (0.04cm);
\filldraw[fill=white, draw=gray] (4.5,8) circle (0.04cm);
\filldraw[fill=white, draw=gray] (5,8) circle (0.04cm);
\filldraw[fill=white, draw=gray] (5.5,8) circle (0.04cm);
\filldraw[fill=white, draw=gray] (6,8) circle (0.04cm);
\filldraw[fill=white, draw=gray] (6.5,8) circle (0.04cm);
\filldraw[fill=white, draw=gray] (7,8) circle (0.04cm);
\filldraw[fill=red, draw=red] (7.5,8) circle (0.04cm);
\filldraw[fill=white, draw=gray] (8,8) circle (0.04cm);
\filldraw[fill=white, draw=gray] (8.5,8) circle (0.04cm);
\filldraw[fill=white, draw=gray] (9,8) circle (0.04cm);
\filldraw[fill=red, draw=red] (9.5,8) circle (0.04cm);
\filldraw[fill=white, draw=gray] (10,8) circle (0.04cm);

\filldraw[fill=red, draw=red] (0,8.5) circle (0.04cm);
\filldraw[fill=red, draw=red] (0.5,8.5) circle (0.04cm);
\filldraw[fill=red, draw=red] (1,8.5) circle (0.04cm);
\filldraw[fill=white, draw=gray] (1.5,8.5) circle (0.04cm);
\filldraw[fill=red, draw=red] (2,8.5) circle (0.04cm);
\filldraw[fill=red, draw=red] (2.5,8.5) circle (0.04cm);
\filldraw[fill=white, draw=gray] (3,8.5) circle (0.04cm);
\filldraw[fill=white, draw=gray] (3.5,8.5) circle (0.04cm);
\filldraw[fill=white, draw=gray] (4,8.5) circle (0.04cm);
\filldraw[fill=white, draw=gray] (4.5,8.5) circle (0.04cm);
\filldraw[fill=white, draw=gray] (5,8.5) circle (0.04cm);
\filldraw[fill=red, draw=red] (5.5,8.5) circle (0.04cm);
\filldraw[fill=red, draw=red] (6,8.5) circle (0.04cm);
\filldraw[fill=white, draw=gray] (6.5,8.5) circle (0.04cm);
\filldraw[fill=red, draw=red] (7,8.5) circle (0.04cm);
\filldraw[fill=white, draw=gray] (7.5,8.5) circle (0.04cm);
\filldraw[fill=red, draw=red] (8,8.5) circle (0.04cm);
\filldraw[fill=red, draw=red] (8.5,8.5) circle (0.04cm);
\filldraw[fill=white, draw=gray] (9,8.5) circle (0.04cm);
\filldraw[fill=red, draw=red] (9.5,8.5) circle (0.04cm);
\filldraw[fill=white, draw=gray] (10,8.5) circle (0.04cm);

\filldraw[fill=white, draw=gray] (0,9) circle (0.04cm);
\filldraw[fill=red, draw=red] (0.5,9) circle (0.04cm);
\filldraw[fill=white, draw=gray] (1,9) circle (0.04cm);
\filldraw[fill=white, draw=gray] (1.5,9) circle (0.04cm);
\filldraw[fill=red, draw=red] (2,9) circle (0.04cm);
\filldraw[fill=white, draw=gray] (2.5,9) circle (0.04cm);
\filldraw[fill=white, draw=gray] (3,9) circle (0.04cm);
\filldraw[fill=white, draw=gray] (3.5,9) circle (0.04cm);
\filldraw[fill=red, draw=red] (4,9) circle (0.04cm);
\filldraw[fill=white, draw=gray] (4.5,9) circle (0.04cm);
\filldraw[fill=white, draw=gray] (5,9) circle (0.04cm);
\filldraw[fill=white, draw=gray] (5.5,9) circle (0.04cm);
\filldraw[fill=white, draw=gray] (6,9) circle (0.04cm);
\filldraw[fill=red, draw=red] (6.5,9) circle (0.04cm);
\filldraw[fill=white, draw=gray] (7,9) circle (0.04cm);
\filldraw[fill=red, draw=red] (7.5,9) circle (0.04cm);
\filldraw[fill=white, draw=gray] (8,9) circle (0.04cm);
\filldraw[fill=white, draw=gray] (8.5,9) circle (0.04cm);
\filldraw[fill=white, draw=gray] (9,9) circle (0.04cm);
\filldraw[fill=red, draw=red] (9.5,9) circle (0.04cm);
\filldraw[fill=white, draw=gray] (10,9) circle (0.04cm);

\filldraw[fill=white, draw=gray] (0,9.5) circle (0.04cm);
\filldraw[fill=red, draw=red] (0.5,9.5) circle (0.04cm);
\filldraw[fill=red, draw=red] (1,9.5) circle (0.04cm);
\filldraw[fill=white, draw=gray] (1.5,9.5) circle (0.04cm);
\filldraw[fill=red, draw=red] (2,9.5) circle (0.04cm);
\filldraw[fill=red, draw=red] (2.5,9.5) circle (0.04cm);
\filldraw[fill=white, draw=gray] (3,9.5) circle (0.04cm);
\filldraw[fill=red, draw=red] (3.5,9.5) circle (0.04cm);
\filldraw[fill=white, draw=gray] (4,9.5) circle (0.04cm);
\filldraw[fill=white, draw=gray] (4.5,9.5) circle (0.04cm);
\filldraw[fill=white, draw=gray] (5,9.5) circle (0.04cm);
\filldraw[fill=red, draw=red] (5.5,9.5) circle (0.04cm);
\filldraw[fill=red, draw=red] (6,9.5) circle (0.04cm);
\filldraw[fill=white, draw=gray] (6.5,9.5) circle (0.04cm);
\filldraw[fill=red, draw=red] (7,9.5) circle (0.04cm);
\filldraw[fill=white, draw=gray] (7.5,9.5) circle (0.04cm);
\filldraw[fill=red, draw=red] (8,9.5) circle (0.04cm);
\filldraw[fill=red, draw=red] (8.5,9.5) circle (0.04cm);
\filldraw[fill=white, draw=gray] (9,9.5) circle (0.04cm);
\filldraw[fill=red, draw=red] (9.5,9.5) circle (0.04cm);
\filldraw[fill=white, draw=gray] (10,9.5) circle (0.04cm);

\filldraw[fill=white, draw=gray] (0,10) circle (0.04cm);
\filldraw[fill=red, draw=red] (0.5,10) circle (0.04cm);
\filldraw[fill=white, draw=gray] (1,10) circle (0.04cm);
\filldraw[fill=white, draw=gray] (1.5,10) circle (0.04cm);
\filldraw[fill=red, draw=red] (2,10) circle (0.04cm);
\filldraw[fill=white, draw=gray] (2.5,10) circle (0.04cm);
\filldraw[fill=white, draw=gray] (3,10) circle (0.04cm);
\filldraw[fill=white, draw=gray] (3.5,10) circle (0.04cm);
\filldraw[fill=red, draw=red] (4,10) circle (0.04cm);
\filldraw[fill=white, draw=gray] (4.5,10) circle (0.04cm);
\filldraw[fill=white, draw=gray] (5,10) circle (0.04cm);
\filldraw[fill=white, draw=gray] (5.5,10) circle (0.04cm);
\filldraw[fill=white, draw=gray] (6,10) circle (0.04cm);
\filldraw[fill=red, draw=red] (6.5,10) circle (0.04cm);
\filldraw[fill=white, draw=gray] (7,10) circle (0.04cm);
\filldraw[fill=red, draw=red] (7.5,10) circle (0.04cm);
\filldraw[fill=white, draw=gray] (8,10) circle (0.04cm);
\filldraw[fill=white, draw=gray] (8.5,10) circle (0.04cm);
\filldraw[fill=white, draw=gray] (9,10) circle (0.04cm);
\filldraw[fill=red, draw=red] (9.5,10) circle (0.04cm);
\filldraw[fill=red, draw=red] (10,10) circle (0.04cm);

\draw[gray,thin] (4.3,1) circle (0.45cm);
\draw[gray,thin] (6.75,0.75) circle (0.45cm);
\draw[gray,thin] (9.4,1.6) circle (0.45cm);
\draw[gray,thin] (3.75,3.25) circle (0.45cm);
\draw[gray,thin] (5.5,4.2) circle (0.45cm);
\draw[gray,thin] (9.23,3.2) circle (0.45cm);
\draw[gray,thin] (1.37,5.87) circle (0.45cm);
\draw[gray,thin] (3.13,6.87) circle (0.45cm);
\draw[gray,thin] (6,6.75) circle (0.45cm);
\draw[gray,thin] (9.25,5.8) circle (0.45cm);
\draw[gray,thin] (8.9,8.1) circle (0.45cm);
\draw[gray,thin] (3.25,8.7) circle (0.45cm);
\draw[gray,thin] (6.63,8.15) circle (0.45cm);
\draw[gray,thin] (1.38,8.15) circle (0.45cm);

\draw[gray,thin] (1.45,2.4) circle (1.35cm);
\draw[-,gray,thin] (1.45,2.4) -- (2.8,2.4);
\node[scale=0.6,gray] at (2,2.2) {$R(n)$};
\node[gray] at (1.45,2.4) [thin]{.};

\end{tikzpicture}
\end{center}
\caption{Illustration of a typical environment $\omega$ in $\Omega_3\cap\Omega_4$ at a large time $n$ in $d=2$. Red sites represent traps and white sites are vacant. The smaller balls with grey boundaries represent the `large' accessible clearings of radius $r(n)=(R_0/2)(\log\rho(n))^{1/d}$. The larger ball with grey boundary represents the `huge' accessible clearing of radius $R(n)=R_0(\log n)^{1/d}-2\sqrt{d}$. Observe that since the balls with grey boundaries represent accessible clearings, they do not contain any red sites which are traps, and moreover each can be reached from the origin by a nearest neighbor path that only visits vacant sites.} \label{figure1}
\end{figure}
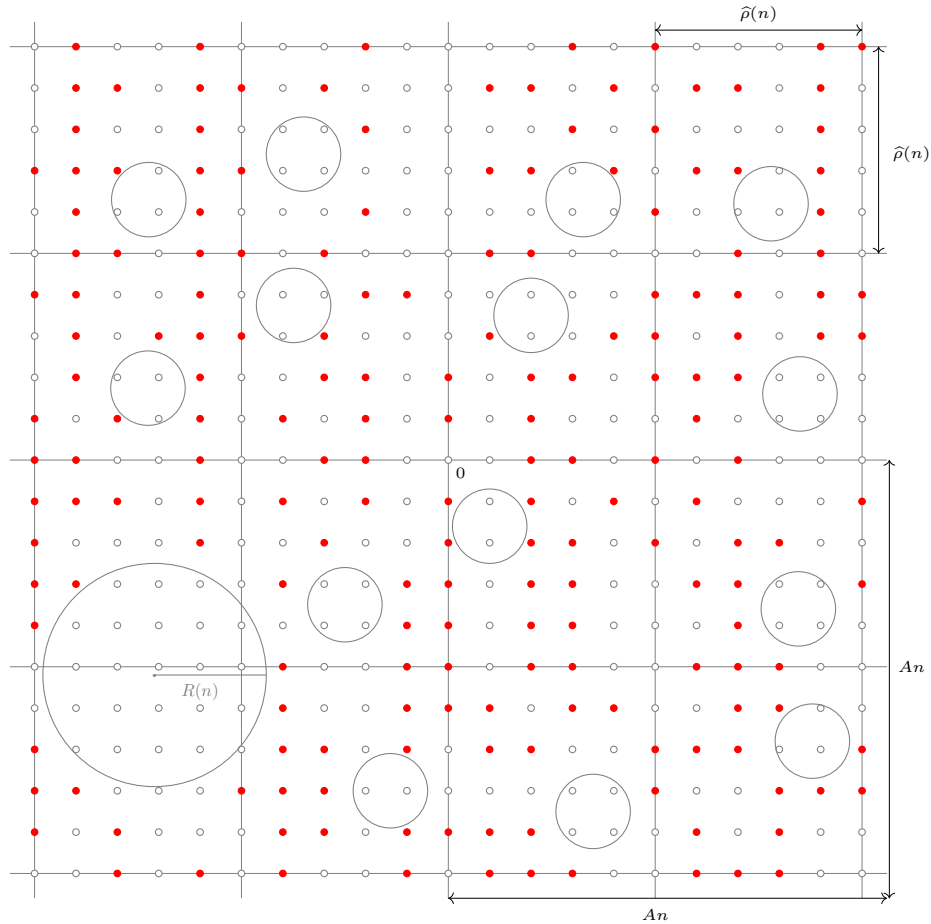

\section{Proof of the upper bound}\label{section5}

Define $\Omega_5:=\{\mathbf{0}\in\mathcal{C}\}\cap\Omega_2\cap\{\omega\in\Omega:\text{\eqref{eqexpectedmass2} holds}\}$ (see \eqref{eqomega2} and Proposition~\ref{prop3}), and let $\omega\in\Omega_5$. Note that $\mathbb{P}(\Omega_5)/\mathbb{P}(\mathbf{0}\in\mathcal{C})=1$. Recall that $\widehat{P}^\omega(\:\cdot\:):=P^\omega(\:\cdot\: \mid S)$, and denote by $\widehat{E}^\omega$ the corresponding expectation. Write 
\begin{equation} \nonumber
\widehat{E}^\omega[N_n]P^\omega(S) = E^\omega\left[N_n \mathbbm{1}_S\right] \leq E^\omega[N_n] .
\end{equation} 
It follows since $\omega\in\Omega_2$ that
\begin{equation}
\widehat{E}^\omega[N_n] \leq \frac{E^\omega[N_n]}{P^\omega(S)} . \label{sect5eq1}
\end{equation}
Let $\varepsilon>0$. By the Markov inequality,
\begin{align} 
\widehat{P}^\omega\left( N_n \geq 2^n \exp\left[\frac{(-k(d,p)+\varepsilon)n}{(\log n)^{2/d}}\right] \right) &\leq \widehat{E}^\omega[N_n]\, 2^{-n} \exp\left[\frac{(k(d,p)-\varepsilon)n}{(\log n)^{2/d}}\right] \nonumber \\
&\leq \frac{1}{P^\omega(S)} \exp\left[\frac{-k(d,p) n}{(\log n)^{2/d}}(1+o(1))\right] \exp\left[\frac{(k(d,p)-\varepsilon)n}{(\log n)^{2/d}}\right] \nonumber \\
& = \exp\left[-\varepsilon n (\log n)^{-2/d} (1+o(1))\right], \label{sect5eq2}
\end{align}
where we have used \eqref{eqexpectedmass2} and \eqref{sect5eq1} in the second inequality. 

We continue with a standard Borel-Cantelli argument. For convenience, define for $n>1$,
$$ Y_n:=(\log n)^{2/d}\left(\frac{\log N_n}{n}-\log 2\right) . $$
Let $\widehat{\Omega}$ be the sample space for the BRW, let
$$\widehat{\Omega}_0:=\{\varpi\in\widehat{\Omega}:\forall\,\varepsilon>0\:\:\exists\,n_0=n_0(\varpi)\:\text{such that}\:\forall\,n\geq n_0,\:Y_n < -k(d,p)+\varepsilon \},$$
and for $n>1$, define
$$A_n:=\left\{ Y_n \geq -k(d,p)+\varepsilon \right\}.$$
By \eqref{sect5eq2}, there exists $n_0\in\mathbb{N}$ and $c(\varepsilon)>0$ such that for all $n\geq n_0$, $\widehat{P}^\omega(A_n)\leq e^{-c(\varepsilon) n(\log n)^{-2/d}}$. Then,
\begin{equation} 
\sum_{n=2}^\infty \widehat{P}^\omega(A_n)=c_1+\sum_{n=n_0}^\infty \widehat{P}^\omega(A_n)\leq c_1+\sum_{n=n_0}^\infty e^{-c(\varepsilon) n(\log n)^{-2/d}}<\infty.  \nonumber
\end{equation}
By the Borel-Cantelli lemma, it follows that $\widehat{P}^\omega(A_n\:\text{occurs i.o.})=0$, where i.o.\ stands for \emph{infinitely often}. Choosing $\varepsilon=1/j$, this implies that for each $j\geq 1$, we have
$$\widehat{P}^\omega(\widehat{\Omega}_j)=1,\quad \widehat{\Omega}_j:=\{\varpi\in\widehat{\Omega}:\exists\,n_0=n_0(\varpi)\:\text{such that}\:\forall\,n\geq n_0,\:Y_n < -k(d,p)+1/j \}. $$
Then, $\widehat{P}^\omega(\widehat{\Omega}_0)=\widehat{P}^\omega(\cap_{j\geq 1}\widehat{\Omega}_j)=1$, which implies that $\mathbb{P}$-a.s.\ on the set $\{\mathbf{0}\in\mathcal{C}\}$,
\begin{equation} 
\underset{n\rightarrow\infty}{\limsup}\, (\log n)^{(2/d)}\left(\frac{\log N_n}{n}-\log 2\right)\leq -k(d,p) \quad\:\: \widehat{P}^\omega\text{-a.s.} \nonumber
\end{equation} 
This completes the proof of the upper bound of Theorem~\ref{thm1}.  \qed

\section{Proof of the lower bound}\label{section6}

The proof of the lower bound of Theorem~\ref{thm1} can be split into four parts, where the first part consists of establishing a quenched environment in which the SLLN will hold. We have completed the first part in Section~\ref{section4} by constructing $\Omega_0$ (see \eqref{eqomega0}). In the remaining parts, we show that the SLLN in \eqref{eqthm1} holds on the set $\Omega_0\cap\{\mathbf{0}\in\mathcal{C}\}$. For each $\omega\in\Omega_0\cap\{\mathbf{0}\in\mathcal{C}\}$ with overwhelming $\widehat{P}^\omega$-probability, we show in part two that at least one particle of the BRW hits a \emph{large} clearing soon enough, and in part three that in the remaining time the desired growth of mass is realized within a \emph{huge} clearing. In part four, we pass to the strong law using a Borel-Cantelli argument.

The pattern of proof, that is, what is shown in each part of the proof, is similar to that of previous works (see \cite{O2021} and \cite{O2024}); nonetheless, the method of proof in part two (see Section~\ref{part2}) is different. Here, we exploit the conditioning on $S_n$ and analyze many ancestral line segments over their respective, possibly overlapping, suitable time intervals. We emphasize that the most challenging part of the proof is part two, where we show that even in the presence of hard traps, with overwhelming probability the BRW is able to travel to a large clearing in which it can produce exponentially many particles. The main difficulty stems from the fact that the process lives in $\mathcal{C}$; therefore, the BRW has no access to possible large trap-free regions outside $\mathcal{C}$ in which the mass could otherwise grow. This also makes the construction of the quenched environment more delicate (see Section~\ref{section4}). 

 
\begin{assumption}
Throughout Section~\ref{section6}, we assume that $\omega\in\Omega_0\cap\{\mathbf{0}\in\mathcal{C}\}$, where $\Omega_0$ is as defined in \eqref{eqomega0} a set of full $\mathbb{P}$-measure. Also, we will refer to an $\omega\in\Omega_0 \cap \{\mathbf{0}\in\mathcal{C}\}$ as a \emph{typical environment}.
\end{assumption}

\subsection{Hitting a large clearing} \label{part2}

We start by making critical use of the conditioning on $S_n$, which is the event of survival of the BRW up to time $n$. The argument below hinges on the fact that production of `few' particles over $[0,n]$ is very unlikely on the event of survival up to $n$. For $n\geq 1$, define the following integer-valued random variables associated to a BRW among Bernoulli traps:
\begin{align}
F_n &:= \#\{ \text{occurrences of fission up to time $n$} \}, \nonumber \\
T_n &:= \#\{ \text{particles that have hit $K=K(\omega)$ up to time $n$} \}, \nonumber \\  
\Sigma_n &:= N_n + T_n ,  \nonumber
\end{align}
where $N_n$ is as before the mass, i.e., number of particles, of the BRW at time $n$. Due to binary branching, it is clear that 
\begin{equation}
\Sigma_n = F_n + 1 .  \label{eq00}
\end{equation}
Note that in the absence of traps, we would have $N_n=F_n+1=2^n$. For a branching random walk $Z=(Z_n)_{n\geq 0}$, define the \emph{range} process $\mathcal{R}=(\mathcal{R}(n))_{n\geq 0}$ as
$$ \mathcal{R}(n) = \bigcup_{0\leq k\leq n}\text{supp}(Z_n) , $$
that is, $\mathcal{R}(n)$ is the set of visited sites up to $n$, and for $n\geq 1$, let the radius of the minimal ball centered at the origin containing $\mathcal{R}(n)$ be
\begin{equation}
M_n = \inf\{r>0:B(0,r)\supseteq \mathcal{R}(n)\}  . \nonumber
\end{equation} 
Recall from \eqref{eqrhot} the definition of $\rho$, and define $r:\mathbb{N}_{\geq n_0}\to\mathbb{R}_+$ by 
\begin{equation} \label{eqrt2}
r(n) = (R_0/2)\left(\log \rho(n)\right)^{1/d} ,
\end{equation}
where $n_0$ is such that $\rho(n)>1$ for all $n\geq n_0$. Next, define for each $\omega\in\Omega$ and $n\geq n_0$, the special subsets of sites
\begin{equation} \label{eqphit}
\Phi(\omega,n)=\{x\in\mathbb{Z}^d : B(x,r(n))\cap\mathbb{Z}^d\subseteq K^c(\omega)\}, \:\:\quad\:\: \widehat{\Phi}(\omega,n)=\Phi(\omega,n)\cap [-An,An]^d 
\end{equation}
in $\mathbb{Z}^d$, where $A>0$ is as in \eqref{eqomega3} and \eqref{eqomega4}. Next, for $n\geq n_0$, define the events
\begin{equation} \label{eqet}
E_n=\left\{\mathcal{R}(n)\cap\widehat{\Phi}(\omega,n)=\emptyset\right\} .  
\end{equation}
Observe that $E_n$ is the event that the BRW does not hit clearings of a certain size within $[-An,An]^d$ up to time $n$. In this part of the proof, we find a suitable upper bound on $P^\omega(E_n \mid S_n)$ for large $n$ in a typical environment $\omega$.

For $0<\theta<1$, we start by conditioning on $\Sigma_{\lfloor\theta n\rfloor}$ as follows. Recall the definition of $k(d,p)$ from \eqref{constant}, choose a constant $k_3$ such that $0<k_3<\theta\, k(d,p)$, and fix this constant. Define for $n>1$, 
\begin{equation} \label{eqct}
\alpha_n := \left\lfloor \exp\left(\frac{k_3 n}{(\log n)^{2/d}}\right)\right\rfloor ,
\end{equation} 
and write 
\begin{align}
P^\omega(E_n\mid S_n) &= \frac{1}{P^\omega(S_n)}
 \left[P^\omega\left(E_n\cap S_n \cap \left\{\Sigma_{\lfloor\theta n\rfloor}\geq \alpha_n\right\}\right) + P^\omega\left(E_n\cap S_n\cap\left\{\Sigma_{\lfloor\theta n\rfloor} < \alpha_n\right\}\right)\right] \nonumber \\  
&\leq \frac{1}{P^\omega(S_n)} \left[P^\omega\left(E_n\cap \left\{\Sigma_{\lfloor\theta n\rfloor}\geq \alpha_n\right\}\right) + P^\omega\left(S_n\cap\left\{\Sigma_{\lfloor\theta n\rfloor} < \alpha_n\right\}\right)\right] . \label{eq40}
\end{align}
Under the law $P^\omega$, since each ancestral line is a random walk up until it hits $K$ if it ever hits $K$, the second term on the right-hand side of \eqref{eq40} can be bounded from above as 
\begin{equation} 
P^\omega\left(S_n\cap\left\{\Sigma_{\lfloor\theta n\rfloor}< \alpha_n\right\}\right) \leq \alpha_n\mathbf{P}_0^\omega\left(\cup_{0\leq k\leq \lfloor\theta n\rfloor}\{X_k\}\cap K=\emptyset\right), \nonumber
\end{equation}
where $(X_k)_{k\geq 0}$ denotes a random walk as before, and we have used that by \eqref{eq00} the process $\Sigma_k$ is nondecreasing in $k$. From the quenched survival asymptotics for a random walk among hard Bernoulli traps in \eqref{eqasymptotics}, for a.e.\ $\omega$ on the set $\{\mathbf{0}\in\mathcal{C}\}$, we then have
\begin{equation} \label{eq41new}
P^\omega\left(S_n\cap\left\{\Sigma_{\lfloor\theta n\rfloor}< \alpha_n\right\}\right) \leq \exp\left[-\frac{(\theta k(d,p)-k_3)n}{(\log n)^{2/d}}(1+o(1))\right] .
\end{equation}
In the rest of the current part of the proof, we find a quenched upper bound that is valid for large $n$ on the first term on the right-hand side of \eqref{eq40}.


Let $\tau_1\leq \tau_2\leq \tau_3\leq \ldots$ denote the fission times of the BRW until time $\lfloor\theta n\rfloor$. Note that this list is almost surely finite for each fixed $n>0$ and that each $\tau_j$ is a stopping time. When more than one fission occurs at the same time; among these occurrences, the indexing of the $\tau_j$ can be done in any arbitrary order. The equality of events $\left\{\tau_{\alpha_n-1}\leq \lfloor\theta n\rfloor\right\}=\left\{\Sigma_{\lfloor\theta n\rfloor}\geq \alpha_n\right\}$ follows from \eqref{eq00}. We start by estimating the first term on the right-hand side of \eqref{eq40} as  
\begin{align}
P^\omega\left(E_n \cap \left\{\Sigma_{\lfloor\theta n\rfloor}\geq \alpha_n\right\}\right)  &= P^\omega(E_n \cap \left\{\Sigma_{\lfloor\theta n\rfloor}\geq \alpha_n\right\} \cap \{M_{\tau_{\alpha_n-1}} > An \})  \nonumber \\
&  \quad + P^\omega(E_n \cap \left\{\Sigma_{\lfloor\theta n\rfloor}\geq \alpha_n\right\} \cap \{M_{\tau_{\alpha_n-1}} \leq An \}) . \label{eq43}
\end{align}
At each time of fission, independently of everything else in the model, randomly label the two offspring as $L$ for left and $R$ for right. For each $j\geq 1$, let $x_j\in\mathbb{Z}^d$ be the position at which $j$th fission occurs, and consider the two ancestral line segments (see Definition~\ref{def2}), both started at the space-time point $(x_j, \tau_j)$; one initiated by the $L$-offspring born at $j$th fission and always continuing with an $L$-offspring at each subsequent time of fission along this ancestral line segment until the minimum of times $\lfloor\theta n\rfloor$ and first hitting to $K$, and the other initiated by the $R$-offspring born at $j$th fission and similarly continuing with $L$-offsprings. Let us label the two aforementioned ancestral line segments as $(L\mathbf{L})_j(\theta n)$ and $(R\mathbf{L})_j(\theta n)$, where $\mathbf{L}$ should be understood as a string of $L$s with length of the string being equal to the number of occurrences of subsequent fissions along the ancestral line segment. If there is no subsequent occurrence of fission along the ancestral line segment, then $\mathbf{L}$ corresponds to an empty string of $L$s. On the event $\left\{\tau_{\alpha_n-1}\leq \lfloor\theta n\rfloor\right\}$, consider the $2(\alpha_n-1)$ ancestral line segments of the type described above:  
\begin{equation}
(L\mathbf{L})_1(\theta n), (R\mathbf{L})_1(\theta n), (L\mathbf{L})_2(\theta n), (R\mathbf{L})_2(\theta n), \ldots, (L\mathbf{L})_{\alpha_n-1}(\theta n), (R\mathbf{L})_{\alpha_n-1}(\theta n). \label{eqancestral} 
\end{equation}
Observe that exactly $\alpha_n$ of them, precisely those where $\mathbf{L}$ in the notation corresponds to an empty string of $L$s, are independent of each other given their starting points. Extend these $\alpha_n$ ancestral line segments, as described in Section~\ref{twocolor}, in a unique way to be defined on the time interval $[0,\tau_{\alpha_n-1}]$ unless they hit $K$ earlier than $\tau_{\alpha_n-1}$. Some of them may have ended earlier due to hitting $K$. Since each of these ancestral line segments moves as a random walk while alive, the first term on the right-hand side of \eqref{eq43} is bounded from above by
\begin{align}
 P^\omega\left( \left\{\Sigma_{\lfloor\theta n\rfloor}\geq \alpha_n \right\} \cap \{M_{\tau_{\alpha_n-1}} > An \}\right) &\leq \alpha_n \mathbf{P}_0^\omega\left(\sup_{0\leq k\leq \lfloor\theta n\rfloor} |X_k|\geq An\right) \nonumber \\
&  \leq \alpha_n \mathbf{P}_0\left(\sup_{0\leq k\leq \lfloor\theta n\rfloor} |X_k|\geq An\right)  \nonumber \\
&  = \alpha_n c_1 \exp\left[-\frac{c_2 A^2 n}{\theta}\right] = \exp\left[-\frac{c_2 A^2 n}{\theta}(1+o(1))\right] , \label{eq44}
\end{align}
where $c_1$ and $c_2$ are positive constants, $\mathbf{P}_0^\omega$ denotes the law under which $X=(X_k)_{k\geq 0}$ is a random walk starting at the origin in the environment $\omega$, we have used the union bound and that $\left\{\Sigma_{\lfloor\theta n\rfloor}\geq \alpha_n\right\}=\left\{\tau_{\alpha_n-1}\leq \lfloor\theta n\rfloor\right\}$ in the first inequality, and used \cite[Proposition 2.1.2]{LL2010} to estimate $\mathbf{P}_0\left(\sup_{0\leq k\leq \lfloor\theta n\rfloor} |X_k| \geq An \right)$.

We now find a quenched upper bound that is valid for large $n$ on 
\begin{equation} \label{eqrealdeal} 
P^\omega(E_n \cap \left\{\Sigma_{\lfloor\theta n\rfloor}\geq \alpha_n\right\} \cap \{M_{\tau_{\alpha_n-1}} \leq An \})    , \end{equation}
which is the second term on the right-hand side of \eqref{eq43}. We emphasize that towards our goal of estimating $P^\omega(E_n \mid S_n)$, thus far we have been carefully `throwing away' some unwanted unlikely events (see \eqref{eq41new} and \eqref{eq44}). The key estimate will be that of \eqref{eqrealdeal}, which is the only probability involving the event $E_n$ and is directly estimated in this section. 

Recall from \eqref{eqrhot} that 
$$ \rho(n)=k_2\frac{n}{(\log n)^{2/d}} , \quad k_2>0 , \quad n\geq 3. $$
Recall that $x_j$ denotes the site at which $j$th fission occurs, and consider the box $x_j+[-\rho(n),\rho(n)]^d$. On the event $\{M_{\tau_{\alpha_n-1}} \leq An \}$, it is clear that $x_j\in[-An,An]^d$ for each $j$, $1\leq j\leq \alpha_n-1$. Then, since $\omega\in\Omega_0\subseteq\Omega_3$, by definition of $\Omega_3$ (see \eqref{eqomega3}), there exists $n_3=n_3(\omega)$ such that for each $n\geq n_3$ and each $j$, there exists $y_j$ in $x_j+[-\rho(n),\rho(n)]^d$ such that $B(y_j,(R_0/2)(\log \rho(n))^{1/d})$ is an accessible clearing. Recall that for $x,y\in\mathcal{C}$, $d_\mathcal{C}(x,y)$ denotes the minimal length of a nearest neighbor path joining $x$ and $y$ such that the entire path is in $\mathcal{C}$. Then, since $\omega\in\Omega_0\subseteq\Omega_1$ (see \eqref{eqomega1}) and both $x_j, y_j\in\mathcal{C}$, it follows from \eqref{distanceestimate} that for each $j$ and all large $n$,  
\begin{equation} \label{eqrealdeal2}
d_\mathcal{C}(x_j,y_j) = d_\mathcal{C}(x_j,y_j)(n)   \leq \lfloor 2\psi d \rho(n)\rfloor  =: h(n).    
\end{equation}
On the event $\left\{\Sigma_{\lfloor\theta n\rfloor}\geq \alpha_n\right\}$, for $1\leq j\leq \alpha_n-1$, consider the ancestral line segments $(L\mathbf{L})_j(\theta n)$ and $(R\mathbf{L})_j(\theta n)$ over the interval $[\tau_j,\lfloor\theta n\rfloor]$ from the list in \eqref{eqancestral}. Restrict them to the period $[\tau_j, \tau_j+h(n)]$ if $\tau_j+h(n)\leq \lfloor\theta n\rfloor$ and extend them to the period $[\tau_j, \tau_j+h(n)]$ if $\tau_j+h(n)>\lfloor\theta n\rfloor$ similarly as before by always choosing the $L$-offspring at each subsequent time of fission along the ancestral line. Note that for $1\leq j\leq \alpha_n-1$, in any case $\tau_j+h(n)\leq n$ for all large $n$ since $0<\theta<1$ is fixed and $\tau_j\leq \lfloor\theta n\rfloor$ on the event $\left\{\Sigma_{\lfloor\theta n\rfloor}\geq \alpha_n\right\}$. Then, since each ancestral line moves as a random walk, by \eqref{eqrealdeal2}, the probability that each of the two ancestral line segments mentioned above starting from site $x_j$ at time $\tau_j$ hits $y_j$ over the period $[\tau_j,\tau_j+h(n)]$ (before hitting $K$) is bounded below by  
\begin{equation} \nonumber
\left(\frac{1}{2d}\right)^{h(n)} \geq e^{-C\rho(n)}, \quad\quad C=C(p,d):= 2\psi d \log (2d)>0 . 
\end{equation}
Since exactly $\alpha_n$ of the ancestral line segments listed in \eqref{eqancestral} are independent of each other given their starting points, for all large $n$,
\begin{equation} \label{eqlowerbound2}
P^\omega(E_n \cap \left\{\Sigma_{\lfloor\theta n\rfloor}\geq \alpha_n\right\} \cap \{M_{\tau_{\alpha_n-1}} \leq An \})\leq 
 \left(1-e^{-C \rho(n)}\right)^{\alpha_n} \leq \exp\left[-e^{-C \rho(n)}\alpha_n\right], 
\end{equation} 
where we have used the elementary estimate $1+x\leq e^x$. Now choose $k_2$ and $k_3$, respectively, in the definitions of $\rho(n)$ and $\alpha_n$ so that $C k_2<k_3$. (So far, we have only required that $k_3<\theta k(d,p)$.) Then, the upper bound in \eqref{eqlowerbound2} can further be bounded from above by
$$  \exp\left[-e^{\frac{(k_3-C k_2)n}{(\log n)^{2/d}}}\right]   , $$
which is superexponentially small in $n$ for large $n$. This, together with the estimates \eqref{eq40}, \eqref{eq41new}, \eqref{eq43} and \eqref{eq44}, imply that for $\omega\in\Omega_0\cap\{\mathbf{0}\in\mathcal{C}\}$,
\begin{align} 
P^\omega(E_n\mid S_n) &\leq \frac{1}{P^\omega(S_n)}\left(\exp\left[-\frac{c_2 A^2 n}{\theta}(1+o(1))\right]+\exp\left[-e^{-\frac{(k_3-C k_2)n}{(\log n)^{2/d}}}\right]\right) \nonumber \\
& \quad + \frac{1}{P^\omega(S_n)} \exp\left[-\frac{(\theta k(d,p)-k_3)n}{(\log n)^{2/d}}(1+o(1))\right] \label{eqfinale}
\end{align}
for all large $n$, where $0<\theta<1$. Recall that $\Omega_0\subseteq\Omega_2$, which implies that $P^\omega(S_n)$ is bounded away from zero for all $n$. Then, \eqref{eqfinale} gives
\begin{equation} \label{eqfinale2}
P^\omega(E_n \mid S_n) \leq \exp\left[-\frac{(\theta k(d,p)-k_3)n}{(\log n)^{2/d}}(1+o(1))\right] .
\end{equation} 
This completes the second part of the proof of the lower bound of Theorem~\ref{thm1}. We emphasize that $k_3<\theta\, k(d,p)$ by choice in the definition of $\alpha_n$ in \eqref{eqct}. In view of \eqref{eqrt2}-\eqref{eqet}, the estimate in \eqref{eqfinale2} says: for each $\omega\in\Omega_0 \cap \{\mathbf{0}\in\mathcal{C}\}$, under the law $P^\omega( \cdot\: \mid S_n)$ with overwhelming probability, at least one particle of BRW hits a \emph{large} clearing within the box $[-An,An]^d$ over $[0,n]$ for large $n$. So far, $A>0$ is arbitrary.

\subsection{The bootstrap argument on growth of mass} \label{bootstrap}

Let $m:\mathbb{N}_+\to\mathbb{N}_+$ be such that 
\begin{equation}
m(n) = o(n(\log n)^{-2/d}) \quad \text{and} \quad \lim_{n\to\infty}\frac{\log n}{\log m(n)}=1  . \label{eqmt1}
\end{equation}  
In what follows, $m(n)$ will be used as a suitable time scale. We will show in succession that in a typical environment, with overwhelming probability for large times the BRW
\begin{itemize}
	\item[(1)] hits a clearing of radius $r(m(n))=(R_0/2)[\log \rho(m(n))]^{1/d}$ and produces at least $e^{\delta_1 m(n)}$ particles with $0<\delta_1<\log 2$ within this large clearing over the period $[0,m(n)]$,
	\item[(2)] hits a clearing of radius $R(m(n))=R_0[\log m(n)]^{1/d}-2\sqrt{d}$ over the period $[m(n),\lfloor(1+\delta_2) m(n)\rfloor]$ with $\delta_2$ small enough,
	\item[(3)] produces sufficiently many particles inside the clearing in (2) over the period $[\lfloor(1+\delta_2) m(n)\rfloor,n]$.
\end{itemize}


\bigskip

\textbf{\underline{Part 1}: Growth inside the large clearing}

\medskip

Recall  $\rho(n)$ from \eqref{eqrhot} and from \eqref{eqrt2} that
\begin{equation} 
r(n) = (R_0/2)\left(\log \rho(n)\right)^{1/d}   \nonumber
\end{equation}
for $n\geq n_0$, where $n_0$ was chosen so that $\rho(n_0)>1$. We will show that for any $0<\delta_1<\log 2$, $\mathbb{P}$-a.s.\ on the set $\{\mathbf{0}\in\mathcal{C}\}$ the event $F_{m(n)}$ occurs with overwhelming $\widehat{P}^\omega$-probability, where 
\begin{equation} \label{eqat2}
F_n := \{\exists\:z_0=z_0(\omega)\in[-An,An]^d\cap\mathbb{Z}^d \:\:\text{such that}\:\: Z_n(B(z_0,r(n)))\geq e^{\delta_1 n}\}.
\end{equation}

Let $0<\delta_1<\log 2$, choose $\lambda\in\mathbb{R}$ satisfying $0<\lambda<1-\delta_1/\log 2$, and set
$$ \lambda_n = \lfloor\lambda n\rfloor .   $$
We will first show that in a typical environment $\omega$, with overwhelming probability a particle of the BRW hits a point, say $z_0$, in $\widehat{\Phi}(\omega,\lambda_n)$ (see \eqref{eqphit}) over the interval $[0,\lambda_n]$, and then the sub-BRW emanating from this particle produces at least $e^{\delta_1 n}$ particles over $[\lambda_n,n]$ inside $B(z_0,r(n))$. Recall from \eqref{eqet} that 
$$ E_n=\{\mathcal{R}(n)\cap\widehat{\Phi}(\omega,n)=\emptyset\}  . $$  
Estimate
\begin{align} 
P^\omega\left(F_n^c \cap S_{\lambda_n}\right) &= P^\omega\left(F_n^c \cap S_{\lambda_n} \cap E_{\lambda_n}\right) + P^\omega\left(F_n^c \cap S_{\lambda_n} \cap E_{\lambda_n}^c\right) \nonumber \\
&\leq P^\omega\left(E_{\lambda_n} \mid S_{\lambda_n} \right) + P^\omega\left(F_n^c \mid E_{\lambda_n}^c\right) . \label{eqnewnew}
\end{align}
By \eqref{eqfinale2}, for each $\omega\in\Omega_0 \cap \{\mathbf{0}\in\mathcal{C}\}$,
\begin{equation} \label{part1eq6}
P^\omega\left(E_{\lambda_n} \mid S_{\lambda_n}\right) \leq \exp\left[-\frac{(\theta k(d,p)-k_3)\lambda_n}{(\log \lambda_n)^{2/d}}(1+o(1))\right].
\end{equation}

Now let $\tau=\tau(\omega)=\inf\{k\geq 0:\mathcal{R}(k)\cap \widehat{\Phi}(\omega,\lambda_n)\neq\emptyset\}$ be the first hitting time of $Z$ to $\widehat{\Phi}(\omega,\lambda_n)$. Let $X_1$ be the ancestral line of $Z$ that first hits $\widehat{\Phi}(\omega,\lambda_n)$, set $z_0=X_1(\tau)$, and denote by $\widetilde{Z}$ the sub-BRW initiated at time $\tau$ by $X_1$. We will work under the law $P^\omega(\:\cdot\: | E_{\lambda_n}^c)=P^\omega(\:\cdot\: | \,\tau\leq\lambda_n)$. By the strong Markov property, $\widetilde{Z}$ is a BRW in its own right. For $n\geq n_0$, let  
\begin{equation} \label{eqlopez}
k:=n-\lambda_n,\quad \widetilde{r}(k)=\frac{R_0}{2}\left[\log \rho\left(\frac{\lambda k}{2(1-\lambda)}\right)\right]^{1/d}, \quad B_k:=B(0,\widetilde{r}(k)).
\end{equation}
For $u\in\mathbb{N}$, let $\big|  \widetilde{Z}_{u}^{B(z_0,\widetilde{r}(k))}  \big|$ denote the number of particles at time $\tau+u$ of $\widetilde{Z}$ whose ancestral lines over $[\tau,\tau+u]$ do not exit $B(z_0,\widetilde{r}(k))$. Observe that $\widetilde{r}(k)\leq r(\lambda_n)$ for all large $k$, which implies $B(z_0,\widetilde{r}(k))\subseteq B(z_0,r(\lambda_n))$, and that $B(z_0,r(\lambda_n))$ is a clearing by definition of $\widehat{\Phi}(\omega,\lambda_n)$. Moreover, $\delta_1 n<(n-\lambda_n)\log 2=k\log 2$ due to the choice $\lambda<1-\delta_1/\log 2$. Next, let 
$$ p_k:=\mathbf{P}_{z_0}(\sigma_{B(z_0,\widetilde{r}(k))}\geq k) = \mathbf{P}_0\left(\sigma_{B_k}\geq k\right) , $$
where, as before, $\mathbf{P}_x$ denotes the law of a generic random walk $(X_k)_{k\geq 0}$ started at $x\in\mathbb{Z}^d$, and $\sigma_A=\inf\{k\geq 0:X_k\notin A\}$ denotes the first exit time of the random walk out of $A$. Now, with $\widetilde{r}(k)$ and $B_k$ as in \eqref{eqlopez}, noting that $B(z_0,\widetilde{r}(k))$ is a clearing and setting $\gamma_k=\exp[-\widetilde{r}(k)]$, Proposition~\ref{prop4} implies that for all large $k$,
\begin{align} 
P^\omega\left(\big| \widetilde{Z}_{n-\tau}^{B(z_0,\widetilde{r}(k))} \big|<e^{-\widetilde{r}(k)} p_k 2^k \,\big|\, E_{\lambda_n}^c \right) &\leq P^\omega\left(\big| \widetilde{Z}_{n-\tau}^{B(z_0,\widetilde{r}(k))} \big|<e^{-\widetilde{r}(n-\tau)} p_{n-\tau} 2^{n-\tau} \,\big|\, E_{\lambda_n}^c \right) \nonumber \\
&\leq \sup_{k\leq j\leq n} P\left(Y_j< e^{-\widetilde{r}(j)} p_j 2^j\right) \leq e^{-c\,\widetilde{r}(k)}, \label{part1eq7}
\end{align} 
where $c$ is a positive constant, $Y_j$ denotes the number of particles at time $j$ of a standard BRW whose ancestral lines up to $j$ do not exit $B_j$, we have used that $e^{-\widetilde{r}(j)} p_j 2^j$ is increasing for all large $j$ and that $n-\tau\geq n-\lambda_n=k$ in in the first inequality, and that $ \widetilde{r}$ is an increasing function in the last inequality. The fact that  $e^{-\widetilde{r}(j)} p_j 2^j$ is increasing for large $j$ follows from \eqref{eqlopez} and \eqref{eq5prop4}, which, respectively, imply that $\widetilde{r}(j)=o(j)$ and $p_j=e^{o(j)}$. We then further conclude that  $e^{\delta_1 n}\leq e^{-\widetilde{r}(k)} p_k 2^k$ for all large $n$ since $\delta_1 n<k\log 2$. Then, by \eqref{part1eq7}, 
\begin{equation} \label{part1eq8}
P^\omega\left(F_n^c \mid E_{\lambda_n}^c\right) \leq P^\omega\left(\big|  \widetilde{Z}_{n-\tau}^{B(z_0,\widetilde{r}(k))}  \big|<e^{-\widetilde{r}(k)}p_k 2^k \big| E_{\lambda_n}^c \right) \leq e^{-c\,\widetilde{r}(k)}, 
\end{equation}
where $F_n$ is as in \eqref{eqat2}. In view of $k=n-\lambda_n$, and recalling $\widetilde{r}$ from \eqref{eqlopez}, we reach the following conclusion via \eqref{eqnewnew}, \eqref{part1eq6} and \eqref{part1eq8}: for each $\omega\in\Omega_0 \cap \{\mathbf{0}\in\mathcal{C}\}$, for all large $n$,
\begin{equation} \nonumber
P^\omega\left(F_n^c \mid S \right)=\frac{P^\omega(F_n^c\cap S)}{P^\omega(S)}\leq c(\omega) P^\omega(F_n^c\cap S) \leq c(\omega) P^\omega(F_n^c\cap S_{\lambda_n}) \leq e^{-c(\log n)^{1/d}}      ,
\end{equation}
where $c=c(d,p)$ is a positive constant, we have used Theorem~\ref{thm2} in the first inequality, and that $S\subseteq S_{k}$ for any $k\geq 0$ in the second inequality. Finally, since $\lim_{n\to\infty}\frac{\log n}{\log m(n)}=1$ by assumption, it follows that for all large $n$,
\begin{equation}\label{part1eq9}
P^\omega\left(F_{m(n)}^c \,\big|\ S \right) \leq e^{-c(\log n)^{1/d}} .    
\end{equation}


\bigskip

\textbf{\underline{Part 2}: Hitting a huge clearing}

\medskip

Let $R:\mathbb{N}_{\geq 2}\to\mathbb{R}$ be defined by
\begin{equation} \label{eqpart21}
R(n)= R_0(\log n)^{1/d}-2\sqrt{d} .
\end{equation}
Observe that $R(n)\asymp R_0(\log n)^{1/d}$ so that a clearing of radius $R(n)$ is a huge clearing in time $n$ (see Definition~\ref{huge}). In this part of the proof, we will work under the law $P^\omega\left(\:\cdot\:\mid F_{m(n)}\right)$. We start by recalling $r(n)$ from \eqref{eqrt2} and that for $0<\delta_1<\log 2$,
\begin{equation} \label{eqat22}
F_{m(n)} := \{\exists\:z_0=z_0(\omega)\in[-Am(n),Am(n)]^d\cap\mathbb{Z}^d \:\:\text{s.t.}\:\: Z_{m(n)}(B(z_0,r(m(n))))\geq e^{\delta_1 m(n)}\} .
\end{equation}
Denote by $\mathbf{e}_1$ the unit vector in the direction of the first coordinate. We will show that for a suitable choice of $\delta_2>0$, $\mathbb{P}$-a.s.\ on the set $\{\mathbf{0}\in\mathcal{C}\}$, the event
\begin{equation} \label{eqgt}
G_n:= \{\exists\: x_0(\omega)\in\mathbb{Z}^d \:\:\text{s.t.}\:\:  Z_{\lfloor(1+\delta_2)m(n)\rfloor}(\{x_0,x_0+\mathbf{e}_1\}) >0 \:\:\text{and}\:\:B(x_0,R(m(n)))\cap\mathbb{Z}^d\subseteq K^c \}   
\end{equation}
occurs with overwhelming $P^\omega\left(\:\cdot\:\mid F_{m(n)}\right)$-probability. We refer the reader to Figure~\ref{figure2} for an illustration of a snapshot at time $m(n)$ related to the strategy described in this part.

It follows from $\omega\in\Omega_0\cap\{\mathbf{0}\in\mathcal{C}\}$ that $\omega\in\Omega_1\cap\Omega_4$. By definition of $\Omega_4$ in \eqref{eqomega4} and since $\lim_{n\to\infty}m(n)=\infty$, there exists $n_4=n_4(\omega)$ such that for all $n\geq n_4$ there exists $x_0=x_0(\omega)\in\mathbb{Z}^d$ such that $B(x_0,R_0(\log m(n))^{1/d}-2\sqrt{d})\cap\mathbb{Z}^d\subseteq [-Am(n),Am(n)]^d\cap \mathcal{C}$. Conditional on $F_{m(n)}$, let $z_0$ be as in \eqref{eqat22} and consider a particle of $Z_{m(n)}(B(z_0,r(m(n))))$. Call this particle $u$ generically, and let $q_u(n)$ be the probability that the sub-BRW initiated by $u$ at time $m(n)$ contributes a particle to the set of sites $\{x_0,x_0+\mathbf{e}_1\}$ at time $\lfloor(1+\delta_2)m(n)\rfloor$, where $\delta_2>0$ is defined as
$$ \delta_2 =    4 A d \psi  . $$
Here, $\psi$ is as in \eqref{distanceestimate}. For a lower bound on $q_u(n)$, neglect the branching of the sub-BRW started by $u$ at time $m(n)$, and consider only a single ancestral line of this sub-BRW. Using that $\omega\in\Omega_1$ (see \eqref{eqomega1}), and that $\max_{y\in\{x_0,x_0+\mathbf{e}_1\}}||X_u(m(n))-y|| \leq 2 d A m(n)$, it then follows that uniformly over all $u$ such that $X_u(m(n))\in B(z_0,r(m(n)))$,
\begin{equation} \label{eqpart22}
q_u(n) \geq \left(\frac{1}{2d}\right)^{\lfloor \delta_2 m(n)\rfloor} =: p(n) .
\end{equation} 
If the sub-BRW started by $u$ at time $m(n)$ hits $\{x_0,x_0+\mathbf{e}_1\}$ before time $\lfloor(1+\delta_2)m(n)\rfloor$, we may suppose that the relevant ancestral line bounces back and forth between $x_0$ and $x_0+\mathbf{e}_1$ until time $\lfloor(1+\delta_2)m(n)\rfloor$. Then, by the Markov property and the independence of particles present at time $m(n)$,
\begin{equation} \nonumber
P^\omega(G_n^c \mid F_{m(n)}) \leq (1-p(n))^{\left\lceil e^{\delta_1 m(n)}\right\rceil} .
\end{equation}
Using the estimate $1+x\leq e^x$ and the expression for $p(n)$ from \eqref{eqpart22}, we conclude that for each $\omega\in\Omega_0 \cap \{\mathbf{0}\in\mathcal{C}\}$, for all large $n$,
\begin{equation} \label{eqpart24}
P^\omega(G_n^c \mid F_{m(n)}) \leq \exp\left[-e^{m(n)(\delta_1-\delta_2\log (2d))}\right] = \exp\left[-e^{m(n)(\delta_1-4Ad\psi\log (2d))}\right].
\end{equation}
Recall that $A>0$ thus far has been arbitrary. Choose $A$ small enough so as to satisfy 
$$  \delta_1-4Ad\psi\log (2d)>0  . $$
Then, by the assumptions in \eqref{eqmt1}, the right-hand side of \eqref{eqpart24} is superexponentially small in $n$.

\medskip

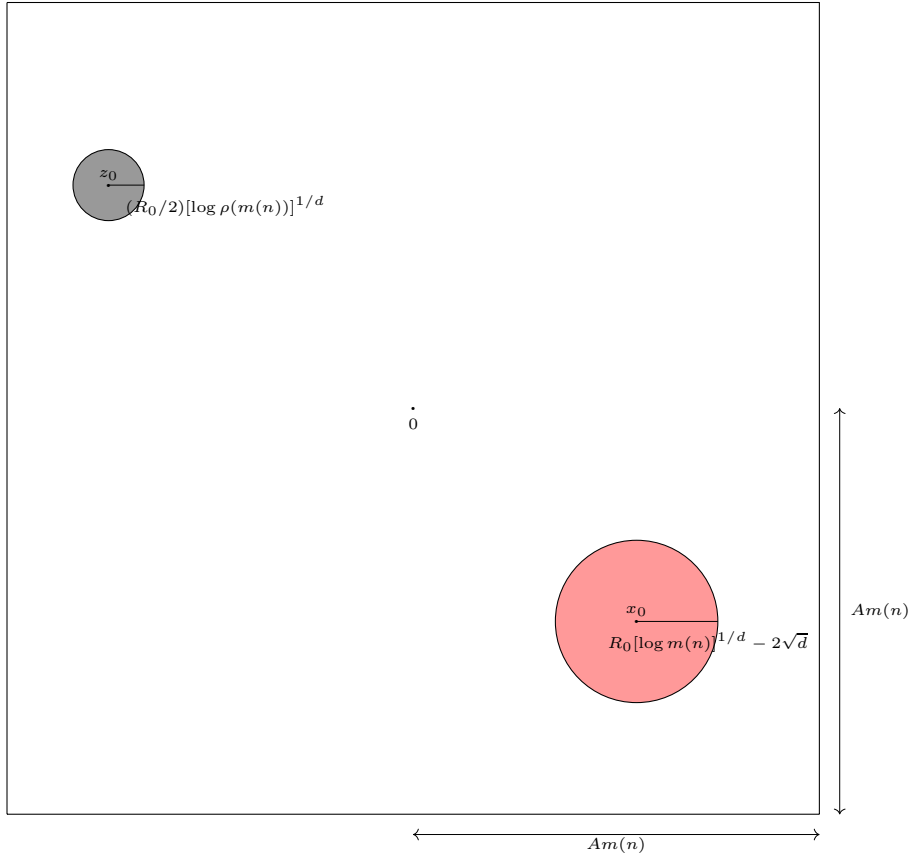
\begin{figure}[ht]
\begin{center}
\begin{tikzpicture}[scale=1.35]
\draw[step=8cm,black,thin] (0,0) grid (8,8);
\draw[<->, very thin] (8.2,0) -- (8.2,4);
\node[] at (8.6,2) {\tiny $A m(n)$};
\draw[<->, very thin] (4,-0.2) -- (8,-0.2);
\node[] at (6,-0.3) {\tiny $A m(n)$};

\node at (4,4) [thin]{.};
\node[] at (4,3.85) {\tiny $0$};

\filldraw[fill=black!40!white, draw=black] (1,6.2) circle (0.35cm);
\draw[-,very thin] (1,6.2) -- (1.35,6.2);
\node[] at (2.15,6) {\tiny $(R_0/2)[\log \rho(m(n))]^{1/d}$};
\node[] at (1,6.2) [thin]{.};
\node[] at (1,6.3) {\tiny $z_0$};
\filldraw[fill=red!40!white,draw=black] (6.2,1.9) circle (0.8cm);
\draw[-,very thin] (6.2,1.9) -- (7,1.9);
\node[] at (6.9,1.7) {\tiny $R_0[\log m(n)]^{1/d}-2\sqrt{d}$};
\node[] at (6.2,1.9) [thin]{.}; 
\node[] at (6.2,2) {\tiny $x_0$};
\end{tikzpicture}
\end{center}
\caption{Illustration of a snapshot at time $m(n)$ of the strategy described in Part $2$. The small grey ball is $B(z_0,r(m(n)))$, which is a `large' accessible clearing and contains at least $e^{\delta_1 m(n)}$ particles at time $m(n)$. The bigger red ball is the `huge' accessible clearing of radius $R(m(n))$ and has at least one particle on its center at time $\lfloor(1+\delta_2)m(n)\rfloor$.} \label{figure2}
\end{figure}

\bigskip

\textbf{\underline{Part 3}: Growth inside the huge clearing}

\medskip

This part of the proof is similar to the corresponding parts in \cite{O2021} and \cite{O2024}, because over the remaining time interval $[\lfloor(1+\delta_2)m(n)\rfloor,n]$, we will only count the particles whose ancestral lines at time $\lfloor(1+\delta_2)m(n)\rfloor$ are in the two-site set $\{x_0,x_0+\mathbf{e}_1\}$, where $x_0$ is as in \eqref{eqgt}, and over $[\lfloor(1+\delta_2)m(n)\rfloor,n]$ do not escape the clearing $B(x_0,R(m(n)))$. Clearly, such particles do not `feel' the presence of traps, and hence their growth does not depend on whether the traps solely suppress the branching, or apply soft or hard killing. Here, with the aim of a strong law in mind, we not only show that the probability of the `unwanted' events go to zero as $n\to\infty$ but also estimate the rate of convergence to zero of these probabilities.

Let $\varepsilon>0$, fix $\alpha>0$, and for $n>1$, define the events
\begin{equation} \label{eqdn}
D_n:=\left\{\exists\:x_0(\omega)\in\mathbb{Z}^d : Z_{n}(B(x_0,R(m(n))\!-\!\alpha))\geq 2^n\exp\left[-\frac{(k(d,p)+\varepsilon)n}{(\log n)^{2/d}}\right]\!,\, B(x_0,R(m(n)))\cap\mathbb{Z}^d\subseteq K^c\right\}  .  
\end{equation}
The constant $\alpha$ in the definition of $D_n$ will play a role in the final part of the proof in Section~\ref{part4}.

Set $c(\omega)=1/P^\omega(S)$, and start with the estimate
\begin{align} \label{part3eq1}
P^\omega(D_n^c \mid S) &\leq c(\omega)\left[P^\omega(D_n^c \cap S \cap G_n) + P^\omega(D_n^c \cap S \cap G_n^c) \right] \nonumber \\
&\leq c(\omega)\left[P^\omega(D_n^c \mid G_n) +P^\omega(S\cap G_n^c\cap F_{m(n)})+ P^\omega(S\cap G_n^c\cap F_{m(n)}^c)   \right] \nonumber \\
&\leq c(\omega)\left[P^\omega(D_n^c \mid G_n) +P^\omega(G_n^c\mid F_{m(n)})+ P^\omega(F_{m(n)}^c\mid S)   \right] .
\end{align}
The second and third terms on the right-hand side of \eqref{part3eq1} can be estimated via \eqref{part1eq9} and \eqref{eqpart24}: for each $\omega\in\Omega_0 \cap \{\mathbf{0}\in\mathcal{C}\}$, for all large $n$,
\begin{equation} \label{part3eq2}
P^\omega\left(G_n^c \mid F_{m(n)}\right) + P^\omega(F_{m(n)}^c \mid S) \leq \exp\left[-e^{-m(n)(\delta_1-4Ad\psi\log 2)}\right] + e^{-c(\log n)^{1/d}} . 
\end{equation}
In the rest of the current part of the proof, we estimate $P^\omega(D_n^c \mid G_n)$.

Conditional on the event $G_n$ (see \eqref{eqgt}), denote by $\widehat{Z}$ a sub-BRW initiated from the two-site set $\{x_0,x_0+\mathbf{e}_1\}$ at time $\lfloor(1+\delta_2)m(n)\rfloor$. By the Markov property, $\widehat{Z}$ is a BRW in its own right. Define $\widehat{R}:\mathbb{N}_{+}\to\mathbb{R}_+$ such that $\widehat{R}(n-\lfloor(1+\delta_2)m(n)\rfloor)=R(m(n))-\alpha$ for all large $n$. (Respecting the conditions in \eqref{eqmt1}, we may and do choose $m(n)$ such that $n-\lfloor(1+\delta_2)m(n)\rfloor$ is increasing on $n\geq n_0$ for some $n_0\in\mathbb{N}$. Therefore, $n_1-\lfloor(1+\delta_2)m(n_1)\rfloor=n_2-\lfloor(1+\delta_2)m(n_2)\rfloor$ implies that $n_1=n_2$ for $n_1\wedge n_2\geq n_0$, where we use $a\wedge b$ to denote the minimum of the numbers $a$ and $b$.) Next, let $k=n-\lfloor(1+\delta_2)m(n)\rfloor$, $\widehat{B}_{k,y}=B(y,\widehat{R}(k))$, and
$$ \widehat{p}_k := \mathbf{P}_{x_0}\left(\sigma_{\widehat{B}_{k,x_0}}\geq k\right) =  \mathbf{P}_{x_0+\mathbf{e}_1}\left(\sigma_{\widehat{B}_{k,x_0+\mathbf{e}_1}}\geq k\right) = \mathbf{P}_{0}\left(\sigma_{\widehat{B}_{k,0}}\geq k\right).  $$
Choose $\alpha\!>\!1$. Then, $\widehat{B}_{k,x_0}\cup\widehat{B}_{k,x_0+\mathbf{e}_1} \subseteq B(x_0,R(m(n)))$, which implies that both $\widehat{B}_{k,x_0}$ and $\widehat{B}_{k,x_0+\mathbf{e}_1}$ are clearings. Proposition~\ref{prop4} with $\gamma(k)=\exp[-\widehat{R}(k)]$ therein then implies that for all large $k$,
\begin{equation} \label{part3eq3}
P^\omega\left(\widehat{Z}_k(\widehat{B}_{k,x_0}\cup\widehat{B}_{k,x_0+\mathbf{e}_1})<e^{-\widehat{R}(k)} \widehat{p}_k 2^k \,\big|\, G_n \right)\leq 
P\left(Y_k< e^{-\widehat{R}(k)} \widehat{p}_k 2^k\right) \leq e^{-c\,\widehat{R}(k)},
\end{equation} 
where $c$ is a positive constant and $Y_k$ denotes the number of particles at time $k$ of a standard BRW whose ancestral lines up to $k$ do not exit $B(0,\widehat{R}(k))$. It follows from \eqref{eq5prop4} that
\begin{equation} \label{part3eq4}
\widehat{p}_k \geq \exp\left[-\frac{\lambda_d k}{(\widehat{R}(k))^2}(1+o(1))\right] = \exp\left[-\frac{k(d,p) (n-\lfloor(1+\delta_2)m(n)\rfloor)}{(\log\,m(n))^{2/d}}(1+o(1))\right],
\end{equation}
where we have used that $k(d,p)=\lambda_d/R_0^2$ and that $\widehat{R}(k)=R(m(n))-\alpha$ together with \eqref{eqpart21}.
It follows from the assumptions on $m(n)$ in \eqref{eqmt1} that  
$$  \frac{n-\lfloor(1+\delta_2)m(n)\rfloor}{(log\, m(n))^{2/d}}  \asymp \frac{n}{(\log n)^{2/d}} , $$
by which we can continue \eqref{part3eq4} with
\begin{equation} \label{part3eq5}
\widehat{p}_k \geq \exp\left[-\frac{k(d,p) n}{(\log n)^{2/d}}(1+o(1))\right]. 
\end{equation}
Then, we have for any $\varepsilon>0$,
\begin{equation} \nonumber
2^n\exp\left[-\frac{(k(d,p)+\varepsilon)n}{(\log n)^{2/d}}\right] = 2^k 2^{\lfloor(1+\delta_2)m(n)\rfloor}\exp\left[-\frac{(k(d,p)+\varepsilon)n}{(\log n)^{2/d}}\right] \leq e^{-\widehat{R}(k)} \widehat{p}_k 2^k
\end{equation}
for all large $n$, where we have used \eqref{part3eq5}, the first assumption on $m(n)$ in \eqref{eqmt1}, and that $\widehat{R}(k)=R(m(n))-\alpha=o(n(\log n)^{-2/d})$  in passing to the inequality. It then follows from \eqref{part3eq3} and the definitions of $G_n$ and $\widehat{Z}$ that for all large $n$,
\begin{equation} \label{part3eq6}
P^\omega(D_n^c \mid G_n) \leq P^\omega\left(\widehat{Z}_k(\widehat{B}_{k,x_0}\cup\widehat{B}_{k,x_0+\mathbf{e}_1})<e^{-\widehat{R}(k)} \widehat{p}_k 2^k \,\big|\, G_n \right) \leq  e^{-c\,\widehat{R}(k)}.
\end{equation}
Collecting the bounds in \eqref{part3eq1}, \eqref{part3eq2} and \eqref{part3eq6}, we conclude that for each $\omega\in\Omega_0 \cap \{\mathbf{0}\in\mathcal{C}\}$, there exists a constant $c_*=c_*(d,p)>0$ (independent of $\varepsilon$ and the environment $\omega$) such that for all large $n$,
\begin{equation}  \label{part3eq7}
P^\omega(D_n^c \mid S) \leq e^{-c_*(\log n)^{1/d}} .
\end{equation}
Finally, simply note that $\left\{N_n<2^n\exp\left[-\frac{(k(d,p)+\varepsilon)n}{(\log n)^{2/d}}\right]\right\} \subseteq D_n^c$.

\subsection{Passing to the strong law} \label{part4}

In this part of the proof, we extend the argument in Part 3 of Section~\ref{bootstrap} to pass from a weak LLN to a strong law. Observe that \eqref{part3eq7} proves that $\mathbb{P}$-a.s.\ on the set $\{\mathbf{0}\in\mathcal{C}\}$,
\begin{equation} \label{part4eq1}
\underset{n\rightarrow\infty}{\liminf}\, (\log n)^{2/d}\left(\frac{\log N_n}{n}-\log 2\right)\geq-k(d,p) \quad\:\: \text{in}\:\:\widehat{P}^\omega\text{-probability}  .
\end{equation}
This would be the lower bound of a type of weak LLN, corresponding to the strong law in Theorem~\ref{thm1}. In this part of the proof, the convergence in \eqref{part4eq1} will be improved to almost sure convergence. We note that in \cite{O2021}, the continuum analogue of the current problem was treated in the case of \emph{mild} obstacles, where the total mass $N_t$ of the branching process was almost surely nondecreasing in $t$ as there was no killing mechanism. In the current model, this important regularity property of $N_n$ is lost, which makes the analysis for a Borel-Cantelli proof in this part more challenging.

Define the function $f:\mathbb{N}_+\to\mathbb{R}_+$ by
\begin{equation} \label{part4eq2}
f(k) = \bigg\lfloor \exp\left\{\left(\frac{2}{c_*}\right)^d(\log k)^d\right\} \bigg\rfloor .
\end{equation} 
Using \eqref{part4eq2}, it can be shown that as $k\to\infty$,
\begin{equation} \label{part4eq3}
f(k+1)-f(k) \sim d\left(\frac{2}{c_*}\right)^d f(k) \frac{(\log k)^{d-1}}{k}.
\end{equation}
Then, setting $g(k)=f(k+1)-f(k)$, one can show using \eqref{eqpart21}, \eqref{part4eq2} and \eqref{part4eq3} that 
\begin{equation}
\lim_{k\to\infty} \frac{g(k) R^2(m(f(k)))}{f(k)} = 0.  \nonumber
\end{equation}
This implies that for any $\varepsilon>0$, for all large $k$,
\begin{equation} \nonumber
f(k+1)-f(k) \leq \frac{\varepsilon f(k)}{4\log 2[\log f(k)]^{2/d}} ,
\end{equation}
and hence that 
\begin{equation} \label{part4eq42}
f(k) \log 2-\frac{k(d,p)+\varepsilon/4}{[\log f(k)]^{2/d}}f(k)  \geq f(k+1) \log 2 -\frac{k(d,p)+\varepsilon/2}{[\log f(k+1)]^{2/d}}f(k+1)
\end{equation}
since $f(z)/[\log f(z)]^{2/d}$ is increasing for large $z$.

We continue by defining two families of events similar to the family $(D_n)_{n>1}$ defined in \eqref{eqdn}. Firstly, by Proposition~\ref{prop1}, keeping the notation and setting $c=1/2=\varepsilon$ therein, there exists $R_c>0$ such that 
\begin{equation} \label{part4eq00}
P\left(Y_n(R_c) = 0 \right) = P\left(Y_n(R_c) < \, \frac{1}{2}\,p_n(R_c)\, 2^n\right) \leq \frac{1}{2} 
\end{equation}
for all $n>R_c$, where $Y_n(R_c)$ denotes the number of particles at time $n$ of a standard BRW whose ancestral lines up to time $n$ do not exit $B(0,R_c)$. Since $\{Y_k(R_c)=0\}$ is an increasing family of events in $k$, it follows from \eqref{part4eq00} that 
\begin{equation} \label{part4eq01}
P\left(Y_k(R_c) = 0 \right) \leq \frac{1}{2} 
\end{equation}
for all $k\geq 1$. For $n\geq 2$, set 
$$ \widetilde{R}(n) = R(m(n)) - R_c.    $$
Now let $\varepsilon>0$, and for $k\in\mathbb{N}_{\geq 2}$, define the events
\begin{equation} \nonumber
\widehat{D}_k:=\left\{\exists\:x_0\in\mathbb{Z}^d:\inf_{n\in [f(k),f(k+1)]\cap\mathbb{Z}} Z_{n}(B(x_0,\widetilde{R}(f(k))+R_c))\geq 2^{f(k+1)}\exp\left[-\frac{(k(d,p)+\varepsilon)f(k+1)}{(\log f(k+1))^{2/d}}\right] \right\}
\end{equation}
and
\begin{align} 
\widehat{D}_k^1:=\bigg\{& \exists\:x_0\in\mathbb{Z}^d : Z_{f(k)}(B(x_0,\widetilde{R}(f(k))))\geq 2^{f(k)}\exp\left[-\frac{(k(d,p)+\varepsilon/4)f(k)}{(\log f(k))^{2/d}}\right]  \nonumber \\
& \:\:\:\text{and}\:\: B(x_0,R(m(f(k))))\cap\mathbb{Z}^d\subseteq K^c \bigg\}  .  \nonumber
\end{align} 
Note that since $2^n\exp\left[-\frac{(k(d,p)+\varepsilon)n}{(\log n)^{2/d}}\right]$ is increasing in $n$ for all large $n$, 
\begin{equation} \label{part4eq0}
\widehat{D}_k \subseteq \left\{N_n\geq 2^n\exp\left[-\frac{(k(d,p)+\varepsilon)n}{(\log n)^{2/d}}\right] \:\:\forall\:n\in[f(k),f(k+1)]\cap\mathbb{Z} \right\}   
\end{equation}
for all large $k$. Therefore, for a Borel-Cantelli proof, we look for a suitable quenched upper bound on $P^\omega(\widehat{D}_k^c \mid S)$ that holds for all large $k$. Start with the estimate
\begin{equation} \label{part4eq7}
P^\omega(\widehat{D}_k^c \mid S) \leq P^\omega(\widehat{D}_k^c \mid \widehat{D}_k^1\cap S) + P^\omega((\widehat{D}_k^1)^c \mid S) \leq \frac{1}{P^\omega(S)} P^\omega(\widehat{D}_k^c \mid \widehat{D}_k^1) + P^\omega((\widehat{D}_k^1)^c \mid S).
\end{equation}
For $\omega\in\Omega_0\cap\{\mathbf{0}\in\mathcal{C}\}$, the second term on the right-hand side of \eqref{part4eq7} can be estimated via \eqref{part3eq7} as
\begin{equation} \label{part4eq8}
P^\omega((\widehat{D}_k^1)^c \mid S) \leq e^{-c_*(\log f(k))^{1/d}} 
\end{equation}
for all large $k$.

Next, we estimate $P^\omega(\widehat{D}_k^c \mid \widehat{D}_k^1)$. For simplicity, set
$$ \varphi(k,\varepsilon) := 2^{f(k)}\exp\left[-\frac{(k(d,p)+\varepsilon/4)f(k)}{(\log f(k))^{2/d}}\right] .  $$ 
Conditional on the event $\widehat{D}_k^1$, let $\mathcal{M}_k$ be the set of $\left\lfloor \varphi(k,\varepsilon) \right\rfloor$ particles inside $B(x_0,\widetilde{R}(f(k)))$ at time $f(k)$ that are closest to the site $x_0$. Let us refer to a particle in $\mathcal{M}_k$ generically as $v$, and let $X_v(f(k))$ be its position at time $f(k)$. Let $p_{c}$ be the probability that at least one ancestral line segment of the sub-BRW initiated by $v$ at time $f(k)$ does not hit $K$ and does not leave $B(X_v(f(k)),R_c)$ over the period $[f(k),f(k+1)]$. Since $B(X_v(f(k)),R_c)\subseteq B(x_0,R(m(f(k))))$ and $B(x_0,R(m(f(k))))$ is a clearing on the event $\widehat{D}_k^1$, \eqref{part4eq01} implies that $p_c\geq 1/2$.

On the other hand, if $\widehat{D}_k^c$ is realized conditional on $\widehat{D}_k^1$, this implies due to \eqref{part4eq42} that at most $(1/4)\left\lfloor \varphi(k,\varepsilon)\right\rfloor$ out of $\left\lfloor \varphi(k,\varepsilon) \right\rfloor$ sub-BRWs initiated at time $f(k)$ by particles in $\mathcal{M}_k$ have at least one particle alive at time $f(k+1)$ with ancestral line contained in $B(X_v(f(k)),R_c)$ and hence contained in $B(x_0,\widetilde{R}(f(k))+R_c)$ throughout the interval $[f(k),f(k+1)]$. We now bound the probability of this event from above. Let $W_0$ be the number of sub-BRWs started by particles in $\mathcal{M}_k$ at time $f(k)$ having at least one particle that is alive at time $f(k+1)$ with ancestral line never leaving $B(X_v(f(k)),R_c)$ over $[f(k),f(k+1)]$. By independence of particles in $\mathcal{M}_k$, $W_0$ is a binomial random variable with $\left\lfloor \varphi(k,\varepsilon) \right\rfloor$ trials and `success probability' $p_c$. For a generic Binomial random variable $W$ with $m$ trials, success probability $p$, and corresponding law $\mathcal{P}$, a standard Chernoff bound yields 
\begin{equation} \label{part4eq9}
\mathcal{P}(W\leq (1-\delta)mp) \leq e^{-\frac{\delta^2 mp}{2}}, \quad \:\: 0<\delta<1 .
\end{equation} 
Now, set $W=W_0$ with $m=\left\lfloor \varphi(k,\varepsilon) \right\rfloor$ and $p=p_c$ as above, and choose $\delta=1-1/(4p_c)$ in \eqref{part4eq9}. With these choices, observe that $\delta>0$ since $p_c\geq 1/2$, and \eqref{part4eq9} implies that
\begin{align} 
P^\omega(\widehat{D}_k^c \mid \widehat{D}_k^1) &\leq P^\omega\left(W_0\leq (1/4)\left\lfloor \varphi(k,\varepsilon) \right\rfloor \mid \widehat{D}_k^1\right)  \nonumber \\
&\leq \exp\left[-\frac{(1-1/(4p_c))^2}{2}p_c   \left\lfloor \varphi(k,\varepsilon)\right\rfloor \right] , \label{part4eq10}
\end{align}  
which is superexponentially small in $k$ since by \eqref{part4eq2} we clearly have 
$$ \lim_{k\to\infty}\frac{\varphi(k,\varepsilon)}{k}=\lim_{k\to\infty}\frac{1}{k}\,2^{f(k)}\exp\left[-\frac{(k(d,p)+\varepsilon/4)f(k)}{(\log f(k))^{2/d}}\right]=\infty  .$$ 
Inserting the bounds in \eqref{part4eq8} and \eqref{part4eq10} into \eqref{part4eq7} yields
\begin{equation} \label{part4eq11}
P^\omega(\widehat{D}_k^c \mid S) \leq e^{-c_*(\log f(k))^{1/d}}
\end{equation}
for all large $k$. Then, by \eqref{part4eq11} and the choice of $f(k)$ in \eqref{part4eq2}, there exist constants $c_0>0$ and $k_0>0$ such that
\begin{equation}
\sum_{k=1}^\infty P^\omega(\widehat{D}_k^c \mid S) \leq c_0 + \sum_{k=k_0}^\infty e^{-c_*(\log f(k))^{1/d}} =c_0 + \sum_{k=k_0}^\infty \frac{1}{k^2} < \infty. \nonumber
\end{equation}     
By the Borel-Cantelli lemma, on a set of full $\widehat{P}^\omega$-measure, only finitely many events $\widehat{D}_k^c$ occur. In view of \eqref{part4eq0}, we arrive at
\begin{equation} \nonumber
\underset{n\to\infty}{\liminf}\, (\log n)^{2/d}\left(\frac{\log N_n}{n}-\log 2\right) \geq -[k(d,p)+\varepsilon] \quad\:\: \widehat{P}^\omega\text{-a.s.} 
\end{equation}
Since this holds for each $\omega\in\Omega_0\cap\{\mathbf{0}\in\mathcal{C}\}$ and $\mathbb{P}(\Omega_0)=1$, this completes the proof of the lower bound of Theorem~\ref{thm1}.



\bibliographystyle{plain}

\end{document}